\documentclass{article}
\usepackage[utf8]{inputenc}
\usepackage{fullpage}

\usepackage{amsmath,amssymb,amsfonts,amsthm}
\usepackage{enumerate}
\usepackage{dsfont}

\usepackage{booktabs}
\usepackage{graphicx}
\usepackage{float}
\usepackage[justification=centering]{caption,subcaption}

\usepackage{tikz}
\usepackage{multirow,hhline}
\usepackage{algorithm,algorithmic}

\usepackage[english]{babel}
\usepackage{csquotes}

\usepackage[sorting=none,url=false,doi=false]{biblatex}
\usepackage[colorlinks,citecolor=blue]{hyperref}

\usepackage[noabbrev,capitalise]{cleveref}

\usepackage{macros}

\title{\textsc{Improved Trial and Error Learning\\
for Random Games}}
\author{
    Jérôme Taupin
    \thanks{Université Paris-Saclay, France}
    \and
    Xavier Leturc
    \thanks{Thales SIX GTS, France}
    \and
    Christophe J. Le Martret
    \footnotemark[2]
}

\begin{document}

\maketitle

\begin{abstract}
When a game involves many agents or when communication between agents is not possible, it is useful to resort to distributed learning where each agent acts in complete autonomy without any information on the other agents' situations. 
Perturbation-based algorithms have already been used for such tasks \cite{TEL,ODL}.
We propose some improvements based on practical observations to improve the performance of these algorithms.
We show that the introduction of these changes preserves their theoretical convergence properties towards states that maximize the average reward and improve them in the case where optimal states exist.
Moreover, we show that these algorithms can be made robust to the addition of randomness to the rewards, achieving similar convergence guarantees.
Finally, we discuss the possibility for the perturbation factor of the algorithm to decrease during the learning process, akin to simulated annealing processes.
\end{abstract}

\section{Introduction}
\label{sec:intro}

Consider a game where multiple agents choose an action among a set of possible actions and observe a reward as an outcome, called utility in this paper.
We are interested in the study of algorithms that aim at finding a choice of strategies that maximizes the average utility of all agents in a distributed manner when playing the same game repeatedly.
By distributed, it is meant that each agent only has knowledge of its own action and its perceived outcome when choosing the next action.
Despite this strong constraint, perturbation-based algorithms such as \ac{tel} \cite{TEL} and \ac{odl} \cite{ODL} are shown to converge to some sense towards action profiles maximizing welfare.
These algorithms rely on a fixed perturbation factor $\epsilon>0$. It is proven \cite{TEL} that when choosing $\epsilon$ close to $0$, the asymptotic behavior of \ac{tel} is to almost only visit Nash equilibria maximizing welfare.
This result is proven by showing that the process induced by \ac{tel} adheres to a kind of perturbed Markov processes, for which the convergence is linked to a notion of potential over an appropriate resistance graph \cite{young93}.

\subsection{Problem Setup}
\label{sec:problem_setup}

Let $I$ be a finite set of $n$ agents. At each round, each agent $i\in I$ selects an action $a_i$ among a finite choice.
Given an action profile $\ba=(a_i)_{i\in I}$ of chosen actions for all agents, each agent $i$ observes a utility $U_i(\ba)$.
As there is a finite amount of action profiles, the set of all possible utilities is finite and we assume without loss of generality that utilities are bounded within $[0,1]$.
The quality of an action profile is measured via its global welfare
\[W(\ba) \eqdef \sum_i U_i(\ba)~.\]
The aim of the algorithms discussed in this paper is to push agents to autonomously act in a way that maximizes this global quantity. Precisely, the proposed algorithm are shown to favor optimal states where all agents have a utility of $1$. If no such configuration exists, they are shown to favor Nash equilibria maximizing the welfare. If there are no Nash equilibria, they are shown to favor states maximizing a trade-off between welfare and a quantity representing the stability of the state.
These statements are formalized by our main results \cref{thm:itel:convergence,thm:ritel:convergence}.

In this paper vectors with coordinates over all agents are written in bold font. Given a vector $\ba$, its $i$-th component is denoted $a_i$ and the remaining components are denoted $a_{-i}$. The vector $(a'_i,a_{-i})$ is the result of replacing $a_i$ with $a'_i$ in the vector $\ba$. Similarly, for a subset $J$ of agents, the corresponding coordinates $a_J$ can be replaced with $a'_J$ to form a new vector $(a'_J,a_{-J})$.
We recall the definition of a Nash equilibrium, which indicates an action profile where no agent can individually change its action to improve its own utility.
\begin{definition}[Nash Equilibrium]
    \label{def:equilibrium}
    An action profile $\ba$ with resulting utilities $\bu$ is a Nash equilibrium---shortened to \emph{equilibrium} in this paper---if for all agent $i$ and action $a'_i\ne a_i$, $U_i(a'_i,a_{-i}) \le u_i$.
\end{definition}

\medskip
The algorithms studied in this paper shall be described as \acp{pmp}. A \ac{pmp} is a family of Markov processes indexed by a perturbation factor $\epsilon\ge0$.
Such process revolves around an homogeneous Markov chain of transition matrix $P^0$ over a finite space $\cX$ which is called the \emph{unperturbed process}, with $P^0_{x,y}$ denoting the probability of transitioning from a state $x$ to another state $y$.
The perturbed process is a homogeneous Markov chain $P^\epsilon$ that behaves in such a way that the perturbation vanishes as the perturbation factor $\epsilon>0$ converges to $0$.
\begin{definition}[Perturbed Markov process]
\label{def:pmp}
A family of homogeneous Markov chains $(P^\epsilon)_{0\le\epsilon<\epsilon_0}$ over a finite state space $\cX$---referred to as $(P^\epsilon)$ for simplicity---forms a \ac{pmp} if for all $x,y\in\cX$, $\lim_{\epsilon\to0} P^\epsilon_{x,y} = P^0_{x,y}$~.
\end{definition}

The convergence of algorithms shall be expressed in terms of stochastically stable states of the underlying \ac{pmp} as defined below.
\begin{definition}[Stochastically stable states]
\label{def:stochastically_stable}
Given a \ac{pmp} $(P^\epsilon)$ with state space $\cX$, we call \acp{sss} the smallest subset $\cS \subset \cX$ such that
the Markov process $(X^\epsilon_k)_{k\ge0}$ with transition matrix $P^\epsilon$ and any initial distribution $X^\epsilon_0$ satisfies
\begin{align}
    \label{eq:def:stochastically_stable}
    \lim_{\epsilon\to0}~\liminf_{k\to\pinfty}~ \PP\bLp X_k^\epsilon\in\cS \bRp = 1~.
\end{align}
In the case where $P^\epsilon$ is aperiodic and irreducible with stationary distribution $\pi^\epsilon$ for all $\epsilon>0$, \cref{eq:def:stochastically_stable} is equivalent to
\begin{align*}
    \lim_{\epsilon\to0}~ \pi^\epsilon_\cS = 1
\end{align*}
where we denote for short $\pi^\epsilon_\cF = \pi^\epsilon(\cF)$ and $\pi^\epsilon_x = \pi^\epsilon(\{x\})$ for any $\cF\subset\cX$ and $x\in\cX$. 
In this case, a state $x$ is a \ac{sss} if and only if
\begin{align*}
    \limsup_{\epsilon\to0}~ \pi^\epsilon_x > 0~.
\end{align*}
\end{definition}

\Cref{eq:def:stochastically_stable} means that for small enough $\epsilon>0$, the algorithm with perturbation $\epsilon$ asymptotically remains in $\cS$ for a fraction of the time arbitrarily close to $1$. The \acp{sss} are the states that are asymptotically visited for a non-negligible time by the process when $\epsilon$ is close to $0$.
The algorithms proposed in this paper are designed so that the \ac{sss} are included in the states that we seek to reach, meaning that the algorithm spends most of the time in these states when $\epsilon$ is small enough.

\subsection{Paper Outline and Contributions}

This paper focuses on variants of the \ac{tel} algorithm. We first propose an improved variant called \ac{itel}, that eliminates sub-optimal behaviors using the fact that utilities are known to be bounded.
We then adapt \ac{itel} to work in the context of random games, resulting in a new algorithm called \ac{ritel}.

\Cref{sec:framework} describes the framework used to study the proposed algorithms.
Precisely, \cref{sec:rpmp} recalls the theory from \cite{young93}, which consists in building a resistance graph over the state space of a \ac{pmp} with a regularity assumption, called \ac{rpmp}. This resistance graph is equipped with a notion of potential and it is shown---see \cref{thm:sss_min_gamma}---that the \acp{sss} of a \ac{rpmp} are exactly the states that minimize this potential in the case where the perturbed process is aperiodic and irreducible. In our context we shall see that either the process is indeed aperiodic and irreducible, or it has well-identified absorbing states that make up the \acp{sss}.
These results were applied to \ac{tel} in \cite{TEL} to identify its \acp{sss}, and we use the same reasoning for \ac{itel} and \ac{ritel}.
To this purpose, \cref{sec:pdl} introduces a general formulation of algorithms designed to fit this theory in the sense that the underlying process is a \ac{rpmp}.
At any point, each agent is associated with a state made up of a mood, a benchmark action and a benchmark utility. The algorithm then consists of an \emph{action policy} describing which action the agent should choose based on its state, and of an \emph{update policy} describing how it should update its state based on the outcome. 

\Cref{sec:itel} introduces the \ac{itel} algorithm and generalizes the convergence statement from \ac{tel} to \ac{itel}.
The differences compared to \ac{tel} essentially boil down to the fact that a utility of $1$ is necessarily optimal and sub-optimal behaviors of \ac{tel} can be removed with this observation.
The policies defining \ac{itel} are given in \cref{tab:itel:policies}.
Informally, it is shown---see \cref{thm:itel:convergence}---that \ac{itel} favors states where all agents have optimal utility, then equilibrium states with maximal welfare, then states maximizing a trade-off between welfare and stability similarly to \ac{tel}.
The main difference lies in the fact that if optimal configurations exist, where all agents achieve a maximal utility of $1$, then the corresponding states are absorbing and \ac{itel} almost surely reaches one such configuration and remains within it.
The proof of convergence of \ac{itel}, which is deferred to \cref{sec:itel:proof}, follows the same overall reasoning as in \cite{TEL}, although some adaptations are required to fit the slight changes and the way arguments are presented may differ slightly.
There exists a closely related algorithm called \ac{odl} \cite{ODL}, which maximizes global welfare without discriminating Nash equilibria.
As for \ac{tel}, it can be improved into an new algorithm \ac{iodl} using the same ideas, which is discussed briefly in \cref{sec:iodl} for completeness.

In the context of random games where agents receive random payoffs and the goal is to maximize the mean welfare, we adapt \ac{itel} in \cref{sec:ritel}  to obtain another algorithm \ac{ritel} that is robust to noisy utilities.
The randomness of utilities introduces several hurdles, which are successfully overcome by having the agents play the same action for periods of $\tau$ rounds to evaluate more accurately the mean utility and by having utilities quantized into bins of size $\delta$.
Provided that both parameters are chosen appropriately, the \acp{sss} of \ac{ritel} are shown to be approximately the same as for \ac{itel} in the deterministic case, see \cref{thm:ritel:convergence}.
The analysis of the \ac{ritel} process requires a slightly weaker notion of regularity that is presented in \cref{sec:ritel:almost_reg}. In particular, \cref{thm:sss_min_gamma} is adapted to fit this larger class of \acp{pmp} in the form of \cref{thm:sss_min_gamma_AR}.

Finally, \cref{sec:gsa} draws a parallel between our approach and generalized simulated annealing. In particular, letting the perturbation factor slowly decrease at an appropriate rate in \ac{itel} yields a stronger sense of convergence.

\bigskip\noindent
This paper is the result of a work done during an internship at Thales SIX GTS in spring 2023.
An application of this work focusing on \ac{itel} and its application to the problem of distributed dynamic channel assignment in wireless clustered networks has already been published, see~\cite{taupinImprovedTrialError2024}.

\section{Theoretical Framework}
\label{sec:framework}

This section introduces the framework necessary to define our algorithms and discusses their convergence properties.

\subsection{Regular Perturbed Markov Processes and Potential}
\label{sec:rpmp}

This section recalls useful results from \cite{young93}.
First, let us describe the regularity condition required on \acp{pmp}.
\begin{definition}[Regular Perturbed Markov Process]
    \label{def:regularity}
    A family $(X^\epsilon)_{0\le\epsilon<\epsilon_0}$ of non-negative real numbers is \emph{regular} if $X^\epsilon=0$ or if there exists $r\ge0$ such that
    \begin{align}
        \label{eq:def:regularity}
        0 < \lim_{\epsilon\to0}\epsilon^{-r}X^\epsilon < \pinfty.
    \end{align}
    $r$ is unique and called the \emph{resistance} of $X^\epsilon$. When $X^\epsilon=0$ it is said by convention that the resistance is $\pinfty$.
    A \ac{pmp} $(P^\epsilon)$ is a \ac{rpmp} if all transition probabilities $P^\epsilon_{x,y}$ are regular.
\end{definition}

Intuitively, the regularity condition implies that $P^\epsilon_{x,y}$ behaves like $\epsilon^r$ for small $\epsilon$.
Given a \ac{rpmp}, one can define a resistance graph over its state space, where directed edges are equipped with the resistance of the associated transition.
The study of the resistance graph allows to describe the behavior of the perturbed process when $\epsilon$ goes to $0$.
\begin{definition}[Resistance Graph]
\label{def:resistance}
Let $(P^\epsilon)$ be a \ac{rpmp}.
\begin{itemize}
    \item The resistance of a transition $P^\epsilon_{x,y}$ is denoted $r(x\to y)$.
    \item Let $\cG$ be the directed graph over the vertex space $\cX$ where there is a directed edge of weight $r(x\to y)$ from $x$ to $y$ when the resistance is finite.
    \item An edge $x\to y$ is said to be an \emph{easy edge} if it minimizes the resistance among all outwards edges from $x$, that is if $r(x\to y) = \min_z r(x\to z)$. The associated minimal resistance is called \emph{outward resistance} and denoted $r^\star(x) = \min_z r(x\to z)$.
    \item The resistance of a path $x\leadsto y$ in $\cG$ is the sum of the resistances of its edges.
\end{itemize}
\end{definition}

The resistance $r(x\to y)$ illustrates how unlikely the transition from $x$ to $y$ becomes as the perturbation factor vanishes.
Notice that as $\epsilon\to0$, $P^\epsilon_{x,y}$ converges to a positive value if and only if $r(x\to y)=0$. Hence a transition exists in the unperturbed process $P^0$ if and only if its resistance is equal to $0$.
Moreover, since all non-zero transition probabilities have finite resistance, $P^\epsilon$ is irreducible for all $\epsilon>0$ if and only if there is a finite-resistance path between any two states in $\cG$.
The tendency of the process to remain in a state $x$ can be measured through a notion of potential. The higher its potential, the more difficult it is to reach the state, or the easier it is to leave.
\begin{definition}[Potential]
\label{def:potential}
\hfill
\begin{itemize}
    \item A $x$-\emph{tree} is a spanning tree of the graph $\cG$ rooted at $x$, \ie, an acyclic sub-graph of $\cG$ such that every vertex $y\ne x$ has a unique outgoing edge and $x$ has none.
    \item The resistance of a $x$-tree is the sum of the resistances of its edges. The tree is said to be minimal if it minimizes the resistance among all $x$-trees.
    \item The \emph{potential}\footnote{The potential $\gamma$ is denoted $\rho$ and called \emph{stochastic potential} in \cite{TEL}.} $\gamma(x)$ of a vertex $x$ is the minimal resistance of a $x$-tree.
\end{itemize}
\end{definition}

Suppose that $P^\epsilon$ is aperiodic and irreducible for all $\epsilon>0$. With this framework, it is then shown that the stationary distributions $\pi^\epsilon$ converge toward a stationary distribution $\pi^0$ of $P^0$ \cite[Lemma 1]{young93} as the perturbation factor $\epsilon$ vanishes. Moreover, $\pi^0_x>0$ if and only if $x$ minimizes the potential $\gamma$ among all states. In other words, such states are exactly the \acp{sss} of the process as defined in \cref{def:stochastically_stable}.

Notice that two states communicate in the unperturbed process $P^0$ if and only if there is a path of zero resistance from one to another and back in the perturbed process. Moreover, a communication class is recurrent in $P^0$ if and only if all edges leaving the class have positive resistance.
Now, consider two states $x$ and $y$ such that $r(x\to y)=0$. A $x$-tree can be modified into a $y$-tree by removing the edge leaving $y$ and adding the edge $x\to y$. Since $r(x\to y)=0$, the obtained $y$-tree has smaller resistance than the original $x$-tree.
It follows that potential is constant over communication classes \cite[Lemma 2]{young93}.
Moreover, a non-recurrent communication class cannot minimize the potential as it has a higher potential then any other class that can be reached from it with a path of zero resistance.

For this reason, an efficient way to compute potentials is to consider an aggregated version of $\cG$ over the communication classes of $P^0$. The resistance of an edge between two classes $X$ and $Y$ is defined as the minimum resistance of a path from any state of $X$ to any state of $Y$. Note that a class of $P^0$ is recurrent if and only if its minimal outward resistance $r^\star$ is nonzero in the aggregated graph. Potential is then defined the same way as in $\cG$.
Since we seek to identify the \acp{sss}, it is not necessary to compute the potential of non-recurrent classes, as they cannot minimize it.
This whole reasoning is detailed in \cite{young93} and eventually leads to the following result.
\begin{theorem}[\cite{young93}, Lemma 1]
    \label{thm:sss_min_gamma}
    Let $(P^\epsilon)$ be a \ac{rpmp} over a finite space $\cX$ such that $P^\epsilon$ is aperiodic and irreducible for all $\epsilon>0$ and denote $\pi^\epsilon$ its stationary distribution.
    Then $\pi^\epsilon$ converges to a stationary distribution $\pi^0$ of $P^0$ as $\epsilon\to0$. Moreover, $\pi^0_x>0$ if and only if $x$ belongs to a recurrence class that minimizes $\gamma$.
    In particular, the set of \acp{sss} is
    \[\cS
    = \cX^\star
    \eqdef \bLa x\in\cX~:~\gamma(x) = \min_\cX\gamma\bRa~.\]
\end{theorem}

From there, the convergence of a \ac{rpmp} is shown by computing its potential and identifying the recurrence classes that minimize it.
In the context of the following algorithms, $P^\epsilon$ either is aperiodic and irreducible so that \cref{thm:sss_min_gamma} applies, or contains absorbing states which are the \acp{sss}.

\subsection{Perturbed Distributed Learning}
\label{sec:pdl}

Let us introduce the notion of \ac{pdl}, which encompasses the general reasoning of all the following algorithms.
The main idea of a \ac{pdl} algorithm is to attribute moods to each agent along with a reference---or benchmark---action and utility. Agents behave differently based on their state, \ie, their mood $m$, benchmark action $\la$ and benchmark utility $\lu$.
In a \emph{discontent} (denoted \discontent{} for short) state, an agent systemically explores a new strategy and may accept it depending on the quality of the outcome.
In a \emph{content} (\content) state, an agent generally keeps playing its benchmark action but may occasionally explore another strategy and accept it if it sees an improvement.
Two intermediate moods, \emph{hopeful} (\hopeful) and \emph{watchful} (\watchful), are specific to \ac{tel} and its variants discussed here and allow to identify Nash equilibria, which is not done in \ac{odl}.
A step of any \ac{pdl} algorithm plays out as follows. Each agent:
\begin{enumerate}
    \item Chooses an action $a$ to play based on an \emph{action policy} taking into account its state $(\lm,\la,\lu)$,
    \item Observes a utility $u$ resulting from the choice of all actions $\ba$,
    \item Updates its state according to an \emph{update policy} taking into account its state $(\lm,\la,\lu)$ along with $a$ and~$u$.
\end{enumerate}
Regarding the update policy, we list here the different behaviors an agent can adopt to update its state $(\lm,\la,\lu)$:
\begin{itemize}
    \item Accept : Go to state $(\content, a, u)$, \ie, accept the outcome and commit to the observed action and utility.
    \item Revert : Go to state $(\content, \la, \lu)$, \ie, remain in the previous state. This is only used from content states.
    \item Reject : Go to state $(\discontent, a, u)$, \ie, reject the outcome and become discontent.
    \item Adopt intermediate mood $m$: Go to state $(m, \la, \lu)$, \ie, keep the same benchmark and wait for another round before making a decision. The intermediate mood $m$ is either \hopeful{} or \watchful.
\end{itemize}
The main structure of any \ac{pdl} algorithm described above is summarized in \cref{fig:PDL_structure}.
A specific \ac{pdl} algorithm is determined by its action and update policies.
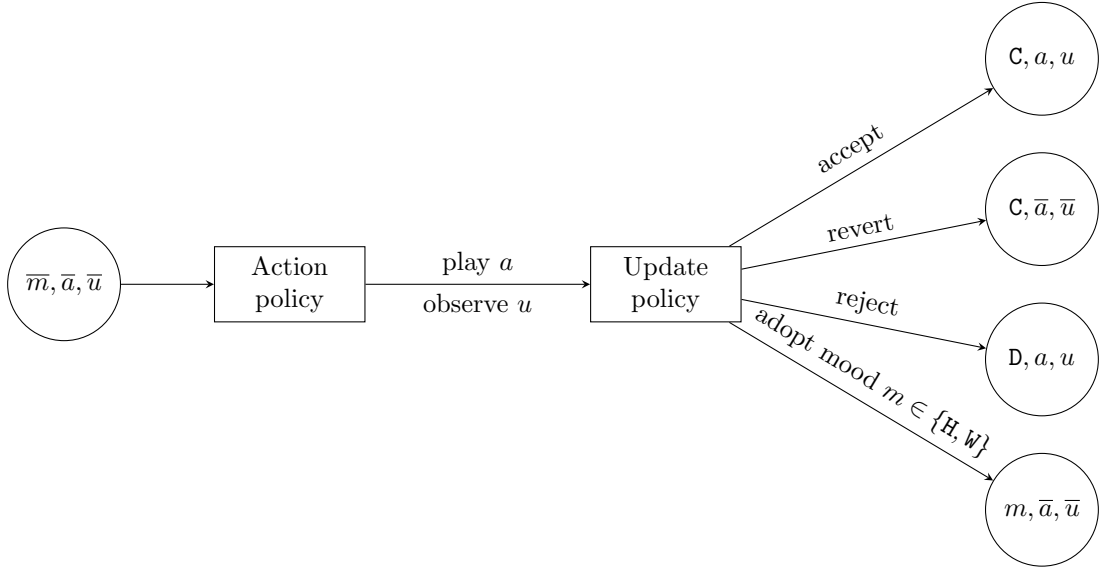
\begin{figure}[ht]
\centering
\begin{tikzpicture}[>=stealth]
    \node[circle, draw, minimum size=1.5cm] (0) at (0,0) {$\lm,\la,\lu$};
    \node[rectangle, draw, minimum width=2cm, minimum height=1cm, align=center] (1) at (3,0) {Action\\policy};
    \node[rectangle, draw, minimum width=2cm, minimum height=1cm, align=center] (2) at (8,0) {Update\\policy};
    \node[circle, draw, minimum size=1.5cm] (b1) at (13,3) {$\cont,a,u$};
    \node[circle, draw, minimum size=1.5cm] (b2) at (13,1) {$\cont,\la,\lu$};
    \node[circle, draw, minimum size=1.5cm] (b3) at (13,-1) {$\disc,a,u$};
    \node[circle, draw, minimum size=1.5cm] (b4) at (13,-3) {$m,\la,\lu$};
    
    \draw[->] (0) -- (1);
    \draw[->] (1) -- (2) node[midway, above] {play $a$} node[midway, below] {observe $u$};
    \draw[->] (2) -- (b1) node[midway, sloped, above] {accept};
    \draw[->] (2) -- (b2) node[midway, sloped, above] {revert};
    \draw[->] (2) -- (b3) node[midway, sloped, above] {reject};
    \draw[->] (2) -- (b4) node[midway, sloped, above] {adopt mood $m\in\{\hope,\watch\}$};
\end{tikzpicture}
\caption{\ac{pdl} Structure.}
\label{fig:PDL_structure}
\end{figure}

In the following, policies are described by tables where each entry describes the decision of an agent with the convention that the agent is in state $(\lm,\la,\lu)$, plays $a$, and observes $u$.
Given $\lm,\la,\lu,a,u$, the decision of the agent is always either deterministic or random involving constant probabilities or probabilities of the form $\epsilon^r$ for some $r\ge0$.
Such policies are therefore functions of $\epsilon$ and designed in a way such that the algorithm behaves as a \ac{rpmp} over the space of all possible states $(\blm,\bla,\blu) = (\lm_i,\la_i,\lu_i)_{i\in I}$.
Indeed, transition probabilities can be expressed as linear combinations of regular quantities, and are therefore regular, satisfying \cref{eq:def:regularity}. Such transition probabilities also converge to a Markov chain which we refer to as the unperturbed process $P^0$ like in \cref{def:pmp}.
From now on we use the \ac{rpmp} theoretical framework of resistances and potentials.

\section{Deterministic Games -- ITEL}
\label{sec:deterministic}

In this section, perceived utilities $U_i(\ba)$ are assumed to be deterministic and bounded within $[0,1]$.
We introduce \ac{itel} in \cref{sec:itel} and analyze its convergence in \cref{sec:itel:results}, leading to our convergence statement in \cref{thm:itel:convergence}.

Let us first discuss a necessary assumption on the underlying game for our analysis to hold.
In a general game, it is possible that an agent $i$ maximizes its utility by playing an action that is detrimental to the other agents. If the other agents have no way to influence $i$ out of playing this action, then it is not possible to design a distributed algorithm that does not get stuck in a sub-optimal state in these situations, without forcing agents to possibly leave a seemingly optimal situation from their point of view. For this reason, the following condition is introduced in \cite{TEL} and is also used in this work.
\begin{assumption}[Interdependence]
\label{assum:interdependence}
The game is \emph{interdependent}, that is, given any action profile $\ba$ and any proper subset $\emptyset\subsetneq J\subsetneq I$ of agents, there exists an agent $i\notin J$ and a choice of actions $a'_J$ such that $U_i(a'_J,a_{-J}) \ne U_i(\ba)$.
\end{assumption}
The interdependence assumption ensures that for any partition of the agents into two groups, there exists a way for the agents of the first group to play that influences the utilities of at least one agent in the second group.

\subsection{Definition of ITEL}
\label{sec:itel}

\ac{itel} is mainly motivated by the following observation: In \ac{tel}, agents that receive an optimal utility of $1$ have a nonzero probability of trying other actions despite the fact that they can only lower their utility by doing so.
The action and update policies of \ac{itel} are given in \cref{tab:itel:policies}.
Given its mood, whether it explored or not, and how it performed compared to its benchmark, the agent behaves according to the corresponding line.
Behavior are either deterministic or a random choice between two behaviors.
Recall that $\discontent$, $\content$, $\hopeful$ and $\watchful$ refer to the moods discussed in \cref{sec:pdl} and that the signification of behaviors ``accept'', ``revert'', etc. are specified in \cref{fig:PDL_structure}.
When exploring, the action is chosen uniformly among all possible actions.

\begin{table}[ht]
    \centering
    
    \begin{subtable}[b]{\textwidth}
        \centering
        \begin{tabular}{||c|c||l||}
            \hhline{|t:===:t|}
            \textbf{Mood} & \textbf{Utility} & \multicolumn{1}{c||}{\textbf{Decision}} \\
            \hhline{||=|=#=||}
            \discontent & / & explore any $a$ \\
            \hhline{||-|-|-||}
            \multirow{2}{*}\content &
            $\lu<1$ & explore $a\ne\la$ with probability $\epsilon$, else play $\la$ \\
            \hhline{||~|-|-||}
            & $\lu=1$ & play $\la$ \\
            \hhline{||-|-|-||}
            \hopeful & / & play $\la$ \\
            \hhline{||-|-|-||}
            \watchful & / & play $\la$ \\
            \hhline{|b:===:b|}
        \end{tabular}
        \caption{Action Policy.}
        \label{tab:itel:action_policy}
    \end{subtable}
    
    \begin{subtable}[b]{\textwidth}
        \centering
        \begin{tabular}{||c|c|c||l||}
            \hhline{|t:====:t|}
            \textbf{Mood} & \textbf{Action} & \textbf{Utility} & \multicolumn{1}{c||}{\textbf{Decision}} \\
            \hhline{||=|=|=#=||}
            \discontent & / & / & accept with probability $\epsilon^{F(u)} \wedge c_F$, else reject \\
            \hhline{||-|-|-|-||}
            \multirow{5}{*}\content &
            \multirow{2}{*}{$a\ne\la$} & $u>\lu$ & accept with probability $\epsilon^{G(\lu,u)}$, else revert \\
            \hhline{||~|~|-|-||}
            & & $u\le\lu$ & revert \\
            \hhline{||~|-|-|-||}
            & \multirow{3}{*}{$a=\la$} & $u>\lu$ & become \hopeful \\
            \hhline{||~|~|-|-||}
            & & $u=\lu$ & revert \\
            \hhline{||~|~|-|-||}
            & & $u<\lu$ & become \watchful \\
            \hhline{||-|-|-|-||}
            \multirow{3}{*}{\hopeful} & \multirow{3}{*}{/} & $u>\lu$ & accept \\
            \hhline{||~|~|-|-||}
            & & $u=\lu$ & revert \\
            \hhline{||~|~|-|-||}
            & & $u<\lu$ & become \watchful \\
            \hhline{||-|-|-|-||}
            \multirow{3}{*}{\watchful} & \multirow{3}{*}{/} & $u>\lu$ & become \hopeful \\
            \hhline{||~|~|-|-||}
            & & $u=\lu$ & revert \\
            \hhline{||~|~|-|-||}
            & & $u<\lu$ & reject \\
            \hhline{|b:====:b|}
        \end{tabular}
        \caption{Update Policy.}
        \label{tab:itel:update_policy}
    \end{subtable}
    
    \caption{\ac{itel} Policies.}
    \label{tab:itel:policies}
\end{table}

In the update policy of \ac{itel}, $F$ is a non-negative non-increasing function of the observed utility $u$ and controls the probability of acceptation from a discontent state.
The probability of accepting an observation from a discontent state is clipped at $c_F\in(0,1)$, which is meant to be close to $1$, for technical purpose. This ensures that there is always a positive probability $1-c_F$ of rejecting the outcome of a discontent observation.
This way, the resistance of accepting $u$ from a discontent state is always equal to $F(u)$ and the resistance of rejecting it is always equal to $0$.
The probability of acceptation from a content state is controlled by $G$ which is a non-negative function of $(\lu,u)$. It is assumed that $G$ is non-decreasing in $\lu$ and non-increasing in~$u$.
The behaviors of agents under \ac{itel} can be described as follows:
\begin{itemize}
    \item A discontent agent explores at random any action. Depending on the outcome $u$, it accepts this action and becomes content with resistance $F(u)$, else remains discontent.
    \item A content agent explores an action different from its benchmark with resistance $1$, unless it already maximizes its benchmark utility. If it does explore, it may accepts the new action only if it performs better than its benchmark utility, and with resistance $G(\lu,u)$.
    \item When a content agent does not explore, it may still notice a different utility than its benchmark due to other agents exploring different actions. In this case, the agent becomes hopeful or watchful depending on the change being an improvement or a deterioration.
    If its utility goes back to normal on the following round, the agent reverts back to its original state. Else, if an improvement is confirmed the agent accepts it and update its benchmark. Finally, if a deterioration is confirmed the agent rejects it and become discontent.
\end{itemize}

Notice that a content agent with optimal utility is bound to its state as it never explores. We call such agents \emph{optimized}, and states were all agents are optimized are referred to as \emph{optimal states}.
This behavior is not present in the original formulation of \ac{tel}, and is the reason for the presence of absorbing states if there exist optimal configurations maximizing the utility for all agents.
Moreover, the functions $F$ and $G$ can be equal to $0$, typically for $u=1$, so that agents observing sufficiently high utilities have no resistance to accept it. This is allowed by the more general form of $G$ as a dual-input function instead of a function of $u-\lu$ as imposed in \cite{TEL}, so that we can have $G(\lu,1)=0$ regardless of $\lu$.
The entire process of the \ac{itel} algorithm is summarized in \cref{alg:itel}.

\begin{algorithm}[ht]
\caption{\ac{itel}}
\label{alg:itel}
\begin{algorithmic}
\STATE
\STATE Initialize at any state $(\blm,\bla,\blv)$
\FOR{iterations $t=1,2,\dots$}
    \STATE $a_i \gets \text{ACTION}(\lm_i,\lv_i)$ according to \cref{tab:itel:action_policy} \textbf{for} $i=1,\dots,n$
    \STATE $u_i \gets \text{SAMPLE}(\ba,i)$ \textbf{for} $i=1,\dots,n$
    \STATE $(\lm_i,\la_i,\lv_i) \gets \text{UPDATE}(\lm_i,\la_i,\lu_i,a_i,u_i)$ according to \cref{tab:itel:update_policy} \textbf{for} $i=1,\dots,n$
\ENDFOR
\end{algorithmic}
\end{algorithm}

\subsection{Theoretical Analysis of ITEL}
\label{sec:itel:results}

We now describe the resistance graph of \ac{itel}.
Recall that its states are triplets $(\blm,\bla,\blu) = (\lm_i,\la_i,\lu_i)_{i\in I}$ describing the states of all agents at a given time in the algorithm. The set of all possible states is denoted $\cX$.
Notice that a discontent agent's behavior does not depend on its benchmark $\la$ and $\lu$, hence we reduce a discontent agent's state to its discontent mood.
In particular, the set of states where every agent is discontent is identified to a single state we denote $D$.
Recall that our goal is to identify the \acp{sss} $\cS$, which are exactly the states $\cX^\star$ that minimize the potential $\gamma$ according to \cref{thm:sss_min_gamma} in the case where the perturbed process $P^\epsilon$ is aperiodic and irreducible.

Due to intermediate moods and discontent agents playing at random, the benchmark utilities of a state do not always correspond to the utilities resulting from the benchmark action profile. The states in which they do correspond are of particular interest as we show that they constitute the recurrent states of the unperturbed process in \cref{prop:itel:recurrence_P_0}.

\begin{definition}[Aligned States]
    \label{def:aligned}
    an agent with benchmark utility $\lu$ is \emph{aligned} with an action profile $\ba$ if $\lu=U(\ba)$.
    A state $x\in\cX$ is said to be aligned if all agents are in a content mood and aligned with the benchmark actions.
\end{definition}

Denote $\cC\subset\cX$ the subset of aligned states.
\ac{itel} is designed so that content agents tend to align with their benchmark action profile when the perturbation factor $\epsilon$ is close to $0$.
\begin{proposition}
    \label{prop:itel:recurrence_P_0}
    The recurrence classes of the unperturbed process $P^0$ are the singletons $\{x\}$ for each aligned states $x\in\cC$, and possibly the communication class of the all-discontent state $D$.
\end{proposition}

\Cref{prop:itel:recurrence_P_0} is proven in \cref{sec:itel:proof_P_0}.
Recall that \ac{tel} is shown to favor Nash equilibria \cite[Theorem 1]{TEL}, for which the definition was recalled in \cref{def:equilibrium}. This is also the case for \ac{itel}.
An aligned state is said to be an equilibrium if its benchmark action profile $\bla$ is itself an equilibrium and we denote $\cE\subset\cC$ the subset of equilibrium states.
We say that an aligned state is optimal if all agents are optimized, \ie, have a benchmark utility of $1$. Since all agents then maximize their utility, such states are also equilibrium states and we denote $\cA\subset\cE$ the set of optimal states.
In an optimal state, all agents keep playing their benchmark action, hence the state is absorbing. We show that these states are in fact the only absorbing states. Moreover, if no such state exist, then the perturbed process is shown to be aperiodic and irreducible as in \ac{tel}.
\begin{proposition}
    \label{prop:itel:recurrence_P_eps}
    Let $\epsilon>0$.
    \begin{itemize}
        \item If $\cA\ne\emptyset$, states in $\cA$ are absorbing for the perturbed process $P^\epsilon$ and the recurrence classes of $P^\epsilon$ are exactly the singletons $\{x\}\subset\cA$.
        \item If $\cA=\emptyset$, $P^\epsilon$ is aperiodic and irreducible.
    \end{itemize}
\end{proposition}

\Cref{prop:itel:recurrence_P_eps} is proven in \cref{sec:itel:proof_res}.
In the first case, we immediately conclude that the \acp{sss} are $\cS = \cA$ and the perturbed process converges \ac{as} to an optimal state.
In the second case, the \ac{rpmp} framework applies and we resort to using \cref{thm:sss_min_gamma} to identify the \acp{sss}.
\Cref{prop:itel:recurrence_P_0} implies that the resistance graph can be reduced to the states $D$ and $\{x\}$ for all $x\in\cC$.
As for \ac{tel}, the computation of resistances and potentials involves welfare and stability of the state, for which we recall the definition.
\begin{definition}[Welfare and Stability]
    \label{def:itel:W&S}
    Let $x=(\blm,\bla,\blu)\in\cC$. We define
    \begin{itemize}
        \item $W(x) \eqdef \sum_i\lu_i$,
        \item $S(x) \eqdef \max \bLa U_i(a_i,\la_{-i})-\lu_i ~|~ \textrm{$a_i\ne\la_i$ such that $U_i(a_i,\la_{-i}) > \lu_i$}\bRa$.
    \end{itemize}
    $S(x)$ is defined only if $x$ is not an equilibrium.
\end{definition}
The welfare $W$ of a state is the welfare of its benchmark action profile, which coincides with the sum of its benchmark utilities if the state is aligned. The stability $S$ of a non-equilibrium state is the largest improvement an agent may observe by changing its action. The higher $S$ is, the likelier an agent is to accept an observation. 
If $F$ and $G$ are respectively chosen as linear function of $u$ and $u-\lu$, the \acp{sss} depend directly on the welfare and stability functions $W$ and $S$ as it is the case for \ac{tel} \cite[Theorem 1]{TEL}.
For more general functions, computing the potentials involves the following quantities.
\begin{definition}[Virtual Welfare and Stability]
    \label{def:itel:tW&tS}
    Let $x=(\blm,\bla,\blu)\in\cC$. We define
    \begin{itemize}
        \item $\tW(x) \eqdef 1 - \sum_i F(\lu_i)$,
        \item $\tS(x) \eqdef 1 - \min \bLa G(\lu_i,U_i(a_i,\la_{-i})) ~|~ \textrm{$i\in I$ and $a_i\ne\la_i$ such that $U_i(a_i,\la_{-i}) > \lu_i$}\bRa$.
    \end{itemize}
    $\tS(x)$ is defined only if $x$ is not an equilibrium.
\end{definition}
In general, $\tW$ is non-decreasing in each individual utility $\lu_i$, hence behaves similarly to the original welfare $W$.
As for $\tS$, greater values indicate a larger possible improvement for an agent when exploring other actions, hence $\tS$ behaves similarly to the original stability $S$.
The particular case where $F$ and $G$ are linear functions is discussed at the end of the section.
In order to establish our result, some bounds on $\tW$ and $\tS$ are required, which are ensured by the following condition on the functions $F$ and $G$.
\begin{condition}[Bounds on resistance functions in \ac{itel}]
    \label{cond:itel}
    There exist constants $F_0$ and $G_0$ such that for all utilities $u$ and $\lu<u$,
    \begin{align*}
        \begin{cases}
            0 \le F(u) \le \frac{F_0}{n}\\
            0 \le G(\lu,u) < G_0\\
            F_0 + G_0 \le 1
        \end{cases}
    \end{align*}
\end{condition}

Under the above condition, it is possible to identify the \acp{sss} of \ac{itel} in the case where there are no optimal states.
\begin{theorem}
    \label{thm:itel:X_star}
     Assume that \cref{assum:interdependence} and \cref{cond:itel} hold and that $\cA=\emptyset$. Then the set $\cX^\star$ of states with minimum potential satisfies:
     \begin{itemize}
        \item If $\cE\ne\emptyset$, $\cX^\star = \arg\max_{x\in\cE} \tW(x)$.
        \item Else, $\cX^\star = \arg\max_{x\in\cC} \tW(x)-\tS(x)$.
    \end{itemize}
\end{theorem}

Intuitively, Nash equilibria are particularly stable states. Indeed, no agent can explore by itself and accept the outcome, as this outcome would necessarily be a deterioration in utility by definition of equilibrium, so that it requires at least two agents to explore to leave such state.
\Cref{thm:itel:X_star} is proven in \cref{sec:itel:proof}. \Cref{prop:itel:recurrence_P_0,prop:itel:recurrence_P_eps} appear as key steps of the proof.
Combining \cref{prop:itel:recurrence_P_eps,thm:itel:X_star,thm:sss_min_gamma} yield the following.

\begin{theorem}[\ac{itel} convergence]
\label{thm:itel:convergence}
    Assume that \cref{assum:interdependence} and \cref{cond:itel} hold. Then:
    \begin{itemize}
        \item The \ac{itel} process converges \ac{as} to an optimal state.
        \item If there is no optimal state, then the \acp{sss} are equilibrium states maximizing $\tW$.
        \item If there is no equilibrium state, then the \acp{sss} are aligned states maximizing $\tW-\tS$.
    \end{itemize}
\end{theorem}

In the last case, $\tW-\tS$ represents a trade-off between welfare and stability.
In the case where $F$ and $G$ are chosen of the form $F:u\mapsto \phi_F - \psi_F\cdot u$ and $G:(\lu,u)\mapsto \phi_G - \psi_G\cdot(u-\lu)$ with $\phi_F$, $\psi_F$, $\phi_G$ and $\psi_G$ positive constants, $\tW$ (resp. $\tS$) is an increasing linear function of $W$ (resp. $S$).
In particular, maximizing $\tW$ is equivalent to maximizing $W$ and maximizing $\tW-\tS$ is equivalent to maximizing $\psi_FW-\psi_GS$.
Indeed, for any state $x\in\cC\setminus\cE$, denote $\blu$ the utilities of $x$ and $\lu_j,u_j$ utilities such that $S(x) = u'-\lu'$. Then
\begin{align*}
    \tW(x) - \tS(x)
    = 1 - n\phi_F + \psi_F\sum_i\lu_i - 1 + \phi_G - \psi_G(u'-\lu')
    = \phi_G - n\phi_F + \psi_FW(x) - \psi_GS(x)~.
\end{align*}

\subsection{ODL and IODL}
\label{sec:iodl}

In this section we briefly discuss how \ac{odl} can be improved in the same way as \ac{tel}. \ac{odl} was introduced in \cite{ODL} and follows the same overall logic as \ac{tel} but without intermediate moods. It is shown that the \acp{sss} are the aligned states maximizing welfare without discriminating for Nash equilibria \cite[Theorem 1]{ODL}.
\ac{odl} however presents several behaviors that seem unintuitive, \eg, content agents can become discontent even after having observed an improvement in utility.
Following the \ac{pdl} framework described in \cref{sec:pdl}, we can define a new algorithm \ac{iodl} that is closer to \ac{itel} without intermediate moods, for which the policies are given in \cref{tab:iodl:policies}.

\begin{table}[ht]
\centering
\caption{\ac{iodl} Policies.}
\label{tab:iodl:policies}

\begin{subtable}[b]{\textwidth}
\centering
\caption{Action Policy.}
\label{tab:iodl:action_policy}
\begin{tabular}{||c|c||l||}
    \hhline{|t:===:t|}
    \textbf{Mood} & \textbf{Utility} & \multicolumn{1}{c||}{\textbf{Decision}} \\
    \hhline{||=|=#=||}
    \discontent & / & explore any $a$ \\
    \hhline{||-|-|-||}
    \multirow{2}{*}\content &
    $\lu\notin A$ & explore $a\ne\la$ with probability $\epsilon$, else play $\la$ \\
    \hhline{||~|-|-||}
    & $\lu\in A$ & play $\la$ \\
    \hhline{|b:===:b|}
\end{tabular}
\end{subtable}

\begin{subtable}[b]{\textwidth}
\centering
\caption{Update Policy.}
\label{tab:iodl:update_policy}
\begin{tabular}{||c|c|c||l||}
    \hhline{|t:====:t|}
    \textbf{Mood} & \textbf{Action} & \textbf{Utility} & \multicolumn{1}{c||}{\textbf{Decision}} \\
    \hhline{||=|=|=#=||}
    \multirow{2}{*}\discontent & \multirow{2}{*}/ & $F(u)>0$ & accept with probability $\epsilon^{F(u)}$, else reject \\
    \hhline{||~|~|-|-||}
    &  & $F(u)=0$ & accept with probability $c_F$, else reject \\
    \hhline{||-|-|-|-||}
    \multirow{5}{*}\content &
    \multirow{2}{*}{$a\ne\la$} & $u>\lu$ & accept with probability $\epsilon^{G(\lu,u)}$, else revert \\
    \hhline{||~|~|-|-||}
    & & $u\le\lu$ & revert \\
    \hhline{||~|-|-|-||}
    & \multirow{3}{*}{$a=\la$} & $u>\lu$ & accept \\
    \hhline{||~|~|-|-||}
    & & $u=\lu$ & revert \\
    \hhline{||~|~|-|-||}
    & & $u<\lu$ & reject \\
    \hhline{|b:====:b|}
\end{tabular}
\end{subtable}

\end{table}

The removal of intermediate moods makes \ac{iodl} unable to discriminate equilibria. The sets $\cC$, $\cA$ and the virtual welfare $\tW$ are the same as in \cref{sec:itel:results}. 
\Cref{prop:itel:recurrence_P_0,prop:itel:recurrence_P_eps} hold for \ac{iodl} and a similar result to \cref{thm:itel:convergence} is shown.
\begin{theorem}[\ac{iodl} convergence]
\label{thm:iodl:convergence}
    Assume that \cref{assum:interdependence} holds and that $0 \le F(u) < \frac1n$. Then:
    \begin{itemize}
        \item The \ac{iodl} process converges \ac{as} to an optimal state.
        \item If there is no optimal state, the \acp{sss} are aligned states maximizing $\tW$.
    \end{itemize}
\end{theorem}

Notice that $G$ has a superficial role theoretically-wise and no particular bound is needed for it.
In the case where $F$ is chosen of the form $F:u\mapsto \phi_F - \psi_F\cdot u$, maximizing $\tW$ is equivalent to maximizing $W$.
\Cref{thm:iodl:convergence} is proven by identifying $\cX^\star$ in a similar way as in \ac{itel}.
The overall proof the same and the few differences are highlighted in \cref{sec:iodl:proof}.

\section{Random Games -- RITEL}
\label{sec:ritel}

In this section we consider random games. For a given action profile $\ba$, the utility  $U_i(\ba)$ is a random variable that is bounded \ac{as} in $[0,1]$. Apart from this assumption, utilities may follow any probability distribution: continuous, discrete or even deterministic.
The introduction of noisy utilities yields the following issues, both from a practical standpoint as it can confuse agents, and from a theoretical standpoint as it prevents us from using the \ac{rpmp} framework.

\paragraph{Instability}
When no agent explores, different payoffs may still be observed due to the random nature of utilities. Under \ac{itel}, those differences would be wrongly interpreted as changes of behaviors from other agents.

\paragraph{Infinite space state}
In \ac{itel}, the proof of convergence is based on the study of the \ac{rpmp} induced by the algorithm. This process acts over the space of all states $(\blm,\bla,\blu)$ of possible moods, actions, and utilities for each agents.
If the support of the random utilities is infinite, which would be the case for continuous distributions, then the state space becomes infinite and the theory of \acp{rpmp} no longer applies.

\paragraph{Non-Regularity}
The randomness of utilities must be taken into account when computing the resistance of paths. It is not true in general that transition probabilities are regular as in \cref{def:regularity}.

\medskip
We introduce \ac{ritel} in \cref{sec:ritel_def}.
\Cref{sec:ritel:almost_reg,sec:ritel:arpmp} adapts the \ac{rpmp} framework discussed in \cref{sec:rpmp} in order to fit the \ac{ritel} process and overcome the issues mentioned above. The \ac{ritel} process is then analyzed in \cref{sec:ritel:results} following the same reasoning as for \ac{itel} in \cref{sec:itel:results} with the necessary adaptations, eventually leading to our convergence statement in \cref{thm:ritel:convergence}.

\subsection{Definition of RITEL}
\label{sec:ritel_def}

In this section we introduce \ac{ritel}, an adaptation of \ac{itel} in the random setup to tackle the issues mentioned above. As for \ac{itel}, we aim at converging to optimal states, then to equilibria with maximal welfare, then to action profiles with maximal trade-off between welfare and stability.
By welfare, we refer to the sum of the mean utilities of all player.

Let us first describe informally the differences between \ac{itel} and \ac{ritel}.
First, iterations are divided in periods of length $\tau$. Within a period, each agent commits to a fixed action so that the chosen action profile remains constant. This way, the variations in utility an agent may observe along the period are solely due to noise, so that it can estimate more reliably the average utility it should receive in this action profile.
When the period ends, each agent compares the estimate of its average utility to its benchmark utility and update their state with a policy similar to \ac{itel}.

Second, in order to consider only a finite amount of possible states, benchmark utilities are restricted to a finite set of bins by slicing the utility range $[0,1]$ into bins $[0,\delta), [\delta,2\delta), \dots, [1-\delta,1), {1}$ for some $\delta>0$ chosen as the inverse of an integer. The highest bin $\{1\}$ may also be referred to as $[1,1+\delta)$ for consistency.
This implies a loss of accuracy that eventually appears in our main result:
\ac{ritel} is not able to discriminate two action profiles when their welfare are too close. Choosing smaller $\delta$ allows finer results, however this requires longer periods to ensure convergence.
Precisely, we shall see that choosing a period length depending on the perturbation factor $\epsilon$ of the form $\tau=\lceil\tau_0 \log(\frac1\epsilon)\rceil$ for some constant $\tau_0$ ensures that the randomness of utilities is negligible in the overall process. The particular choice of $\tau_0$ is guided by a condition which essentially boils down to the fact that better accuracy, \ie, smaller $\delta$, requires higher reliability on the samples, \ie, higher~$\tau_0$.

Third, when comparing utilities to choose a behavior, an agent actually compares the utility bins. Furthermore, it is important for stability to avoid the situation where an agent whose mean utility is close to the border between two bins constantly switches between both bins. For this reason, agents consider a change of utility significant enough to act on it only when the new utility is neither in the benchmark bin, nor in an adjacent bin, which implies a gap in utilities of at least $\delta$.

\subsubsection{Notations}

In the following, actual utilities are denoted using the letter $u$ whereas utility bins are denoted using the letter $v$. We denote $|\cdot|$ the function mapping a utility $u$ to the bin $v = |u|$ it belongs to.
$v^-$ and $v^+$ denote the two edges of a bin $v$, \ie, $v=[v^-,v^+)$ with $v^+=v^-+\delta$.
Given a bin $v=[v^-,v^+)$, we denote $v+k\delta = [(v+k\delta)^-,(v+k\delta)^+) = [v^-+k\delta,v^++k\delta)$ for any integer $k$ and use comparison operators $=$, $\le$, $\ge$, $<$, and $>$ using the natural order on the bins.
We write $v'=v\pm\delta$ (resp. $v'\ne v\pm\delta$) if $v'\in\{v-\delta,v,v+\delta\}$ (resp. $v'\notin\{v-\delta,v,v+\delta\}$).

Given any action profile $\ba$, we denote $M_i(\ba) = \EE[U_i(\ba)]$ the mean utility of agent $i$ in the action profile, and $N_i(\ba) = |M_i(\ba)|$ the corresponding bin.
During one step of the \ac{ritel} algorithm, an agent in a state $(\lm,\la,\lv)$ plays an action $a$ for $\tau$ steps and observes i.i.d. samples $\tilde u_1, \tilde u_2, \dots, \tilde u_\tau$ of distribution $U(\ba)$. The agent then computes its average utility
\[u \eqdef \frac1\tau \sum_{k=1}^\tau \tilde u_k\]
and the corresponding bin $v=|u|$. We denote $U^{(\tau)}$ the distribution of $u$.
Mean utilities are denoted $\lmu = \EE[U(\bla)]$ and $\mu = \EE[U(\ba)]$. Their respective bins are denoted $\lnu = |\lmu|$ and $\nu = |\mu|$, and referred to as mean bins.
From now on, the term ``offset observation'' always refers to an agent observing an utility bin $v\ne\nu\pm\delta$ where $\nu$ is the mean bin of the agent in the played action profile. In other words, we consider an observed average utility to be offset if it does not belong to its mean bin or a bin next to it, which implies a deviation of the empirical mean of at least $\delta$ from the true mean.

Finally, $F$ and $G$ are introduced with the same purpose as in \ac{itel}, although they are used as functions of bins instead of true utilities. We define them over bins by applying them to the lower bounds of the bins, that is $F(v) = F(v^-)$ and $G(\lv,v) = G(\lv^-,v^-)$.

\subsubsection{Algorithm}

The policies of \ac{ritel} are detailed in \cref{tab:ritel:policies} and the algorithm is summarized in \cref{alg:ritel}.
Notice that the policies are essentially the same as in \ac{itel}, except that agents make decision based on bins instead of utilities. They also have a small tolerance to change: observing bins $\lv-\delta, \lv, \lv+\delta$ is interpreted as if no change had happened, whereas observing bins $\lv+2\delta, \lv+3\delta, \dots$ is interpreted as an improvement and observing bins $\lv-2\delta, \lv-3\delta, \dots$ is interpreted as a deterioration.

\begin{table}[ht]
\centering
\caption{\ac{ritel} Policies.}
\label{tab:ritel:policies}

\begin{subtable}[b]{\textwidth}
\centering
\caption{Action Policy.}
\label{tab:ritel:action_policy}
\begin{tabular}{||c|c||l||}
    \hhline{|t:===:t|}
    \textbf{Mood} & \textbf{Utility} & \multicolumn{1}{c||}{\textbf{Decision}} \\
    \hhline{||=|=#=||}
    \discontent & / & explore any $a$ \\
    \hhline{||-|-|-||}
    \multirow{2}{*}\content &
    $\lv<\{1\}$ & explore $a\ne\la$ with probability $\epsilon$, else play $\la$ \\
    \hhline{||~|-|-||}
    & $\lv=\{1\}$ & play $\la$ \\
    \hhline{||-|-|-||}
    \hopeful & / & play $\la$ \\
    \hhline{||-|-|-||}
    \watchful & / & play $\la$ \\
    \hhline{|b:===:b|}
\end{tabular}
\end{subtable}

\begin{subtable}[b]{\textwidth}
\centering
\caption{Update Policy.}
\label{tab:ritel:update_policy}
\begin{tabular}{||c|c|c||l||}
    \hhline{|t:====:t|}
    \textbf{Mood} & \textbf{Action} & \textbf{Utility} & \multicolumn{1}{c||}{\textbf{Decision}} \\
    \hhline{||=|=|=#=||}
    \discontent & / & / & accept with probability $\epsilon^{F(u)} \wedge c_F$, else reject \\
    \hhline{||-|-|-|-||}
    \multirow{5}{*}\content &
    \multirow{2}{*}{$a\ne\la$} & $v\ge\lv+2\delta$ & accept with probability $\epsilon^{G(\lv,v)}$, else revert \\
    \hhline{||~|~|-|-||}
    & & $v\le\lv+\delta$ & revert \\
    \hhline{||~|-|-|-||}
    & \multirow{3}{*}{$a=\la$} & $v>\lv+\delta$ & become \hopeful \\
    \hhline{||~|~|-|-||}
    & & $v=\lv\pm\delta$ & revert \\
    \hhline{||~|~|-|-||}
    & & $v\le\lv-2\delta$ & become \watchful \\
    \hhline{||-|-|-|-||}
    \multirow{3}{*}{\hopeful} & \multirow{3}{*}{/} & $v\ge\lv+2\delta$ & accept \\
    \hhline{||~|~|-|-||}
    & & $v=\lv\pm\delta$ & revert \\
    \hhline{||~|~|-|-||}
    & & $v\le\lv-2\delta$ & become \watchful \\
    \hhline{||-|-|-|-||}
    \multirow{3}{*}{\watchful} & \multirow{3}{*}{/} & $v\ge\lv+2\delta$ & become \hopeful \\
    \hhline{||~|~|-|-||}
    & & $v=\lv\pm\delta$ & revert \\
    \hhline{||~|~|-|-||}
    & & $v\le\lv-2\delta$ & reject \\
    \hhline{|b:====:b|}
\end{tabular}
\end{subtable}

\end{table}

\begin{algorithm}[ht]
\caption{\ac{ritel}}
\label{alg:ritel}
\begin{algorithmic}
\STATE
\STATE Initialize at any state $(\blm,\bla,\blv)$
\FOR{periods $t=1,2,\dots$}
    \STATE $a_i \gets \text{ACTION}(\lm_i,\lv_i)$ according to \cref{tab:ritel:action_policy} \textbf{for} $i=1,\dots,n$
    \FOR{$k=1,\dots,\tau$}
        \STATE $\tilde u_{i,k} \gets \text{SAMPLE}(\ba,i)$ \textbf{for} $i=1,\dots,n$
    \ENDFOR
    \STATE $u_i \gets \frac{1}{\tau}\sum_{k=1}^\tau \tilde u_{i,k}$ \textbf{for} $i=1,\dots,n$
    \STATE $v_i \gets |u_i|$ \textbf{for} $i=1,\dots,n$
    \STATE $(\lm_i,\la_i,\lv_i) \gets \text{UPDATE}(\lm_i,\la_i,\lv_i,a_i,v_i)$ according to \cref{tab:ritel:update_policy} \textbf{for} $i=1,\dots,n$
\ENDFOR
\end{algorithmic}
\end{algorithm}

\subsubsection{Theoretical Framework}

Let us discuss the conditions needed to use the \ac{rpmp} framework given by \cref{def:pmp,def:regularity}.
Since the states of the algorithm are updated at the end of each period, we define the perturbed process $P^\epsilon$ to model the transitions from one state to another after a full period of sampling, for which we recall the length $\tau=\lceil\tau_0 \log(\frac1\epsilon)\rceil$.
Notice that $P^\epsilon$ also corresponds to the process without periods---or equivalently with a period length of $1$---but where each utility distribution $U$ is replaced with the distribution $U^{(\tau)}$ of its empirical average.
Indeed, the only role of the periods is to reduce the variance of observations.
When $\epsilon\to0$, $\tau\to\pinfty$ and $U^{(\tau)}$ converges as a distribution towards $\mu$.

The unperturbed process $P^0$ can therefore be defined by taking the limit as $\epsilon\to0$ in \cref{tab:ritel:policies} and replacing each observed utility $U^{(\tau)}$ with $\mu$.
Formally, let us check that the transition probabilities indeed converge to such process.
Consider an agent in a state $(\lm,\la,\lv)$ who observes $u \sim U^{(\tau)}$, the probability that the agent chooses a given behavior is a combination of probabilities of $u$ belonging to specific bins and of probabilities of choosing a behavior given the observed bin.
For example, the probability that a content agent who explored accepts its exploration with new benchmark bin $v$ is $\PP\bLp U^{(\tau)}\in v \bRp \epsilon^{G(\lv,v)}$.
The probability that a hopeful agent becomes watchful is $\sum_{v\le\lv-2\delta} \PP\bLp U^{(\tau)}\in v \bRp$.

Since $\lim_{\epsilon\to0}\tau(\epsilon) = \pinfty$, the law of large numbers implies that as $\epsilon\to0$, the quantities $\PP(U^{(\tau(\epsilon))}\in v)$ converge to $1$ when $\mu=\EE[U]\in v$, else converge to $0$, except for one case. Indeed, if $U$ is not deterministic and its mean $\mu=v^-$ is exactly on the edge between two bins $v$ and $v-\delta$, then both bins are observed with probability $\frac12$ when $\tau\to\pinfty$ as the distribution of the empirical mean becomes symmetric due to the central limit theorem.
Regardless of the case, it always holds that $\lim_{\epsilon\to0} \PP\bLp U^{(\tau(\epsilon))}\in v\bRp$ exists, which in turn implies that all transition probabilities $P^\epsilon_{x,y}$ converge to a limit $P^0_{x,y}$ as $\epsilon\to0$.
Moreover, in $P^0$ the observed utilities are deterministic or equally distributed between two adjacent bins in the special case discussed above.

Regarding the regularity condition given in \cref{def:regularity}, we can see from the above discussion that this condition would be satisfied if and only if $\PP\bLp U^{(\tau(\epsilon)}\in v \bRp$ is regular for all random utilities $U$ and all bins $v$.
This is unfortunately not true in general. However, a slightly weaker condition is satisfied by any distribution.
The following two sections are devoted to this discussion. \cref{sec:ritel:almost_reg} introduces a weaker notion of regularity that is satisfied by $\PP\bLp U^{(\tau(\epsilon)}\in v \bRp$ using large deviation results. \cref{sec:ritel:arpmp} adapts the theory of \acp{rpmp} to fit this new regularity condition and yields a similar result to \cref{thm:sss_min_gamma}, see \cref{thm:sss_min_gamma_AR}.
From there we are able to analyze the convergence of \ac{ritel} in \cref{sec:ritel:results}.

\subsection{Almost Regularity and Large Deviations}
\label{sec:ritel:almost_reg}

In this section we explain why for any utility distribution $U$ and bin $v$, the probability $\PP(U^{(\tau(\epsilon)}\in v)$ satisfies the following definition, which follows from Cramér's theorem.

\begin{definition}[Almost Regularity]
    \label{def:almost_regularity}
    A family $(X^\epsilon)_{0\le\epsilon<\epsilon_0}$ of non-negative real numbers is \emph{almost regular} if there exists $r\ge0$ such that
    \begin{align}
    \label{eq:def:almost_regularity}
    \begin{cases}
        \lim_{\epsilon\to0}\epsilon^{-r'}X^\epsilon = 0 & \text{for all $r'<r$,}\\
        \lim_{\epsilon\to0}\epsilon^{-r'}X^\epsilon = \pinfty & \text{for all $r'>r$.}
    \end{cases}
    \end{align}
    $r$ is unique and called the \emph{resistance} of $X^\epsilon$. $r$ can be equal to $\pinfty$, which is the case in particular when $X^\epsilon=0$.
\end{definition}

It is immediate that regularity as defined in \cref{def:regularity} implies almost regularity, and that both notions of resistance coincide in this case.
Almost regularity suggests that $X^\epsilon$ behaves somewhat like $\epsilon^r$ but in a broader sense.
For instance, the sequence $\bLp\log(\frac{1}{\epsilon}\bRp^\alpha \epsilon^r)_{\epsilon>0}$ is almost regular with resistance $r$ regardless of $\alpha\in\RR$, although $\log(\frac{1}{\epsilon})^\alpha$ may converge to $0$ or to $\pinfty$.
In particular, if the resistance of $X^\epsilon$ is positive then $X^\epsilon \to 0$ as $\epsilon\to0$ by applying the definition with $r'=0$.
Let us now recall Cramér's theorem.
For a real-valued random variable $U$, define its \ac{lmgf}
\begin{align*}
\Lambda_U:t \mapsto \log(\EE[\exp(tU)])
\end{align*}
along with its Legendre transform
\begin{align*}
    \Lambda^*_U:x \mapsto \sup_{t\in\RR}(tx-\Lambda_U(t))~.
\end{align*}
Recall a few elementary properties of the function $\Lambda^*_U$ along with Cramér's theorem.
\begin{lemma}
\label{lemma:Lambda_props}
The function $\Lambda^*_U$ takes the value $0$ at the mean $\mu$ of $U$. It is also decreasing over $(\minfty,\mu]$ and increasing over $[\mu,\pinfty)$.
Moreover, if $x>\mu$ then $\PP(U\ge x)=0$ if and only if $\Lambda^*_U(x)=\pinfty$, and if $x<\mu$ then $\PP(U\le x)=0$ if and only if $\Lambda^*_U(x)=\pinfty$.
\end{lemma}

\begin{lemma}[Cramér {\cite[Theorem 23.3]{KlenkeProbability2020}}]
    \label{lemma:cramer}
    Let $U$ be any non-degenerate random variable with expectation $\mu$ and finite \ac{lmgf}.
    Then, the empirical mean $U^{(\tau)}$ satisfies a large deviation principle with rate function $\Lambda^*_U$, that is:
    \begin{itemize}
        \item For all $x>\mu$,
        \begin{align*}
            \lim_{\tau\to\pinfty} \frac{1}{\tau}\log\big(\PP(U^{(\tau)} \ge x)\big) = -\Lambda^*_U(x)~.
        \end{align*}
        \item For all $x<\mu$,
        \begin{align*}
            \lim_{\tau\to\pinfty} \frac{1}{\tau}\log\big(\PP(U^{(\tau)} < x)\big) = -\Lambda^*_U(x)~.
        \end{align*}
    \end{itemize}
\end{lemma}
In \cref{lemma:cramer}, writing $\PP(U^{(\tau)} \ge x)$ or $\PP(U^{(\tau)} > x)$ does not influence the result, and similarly for the lower tail. The inequalities are stated in a way that is compatible with the form of the bins as $[v^-,v^+)$.
Using \cref{lemma:cramer}, the probability of observing a given bin is made almost regular by having $\tau$ diverge to $\pinfty$ when $\epsilon\to0$ at an appropriate rate.
\begin{proposition}
    \label{prop:almost_regular_deviation}
    Let $\tau(\epsilon) = \lceil\tau_0 \log(\frac1\epsilon)\rceil$ for some constant $\tau_0>0$. Under the same conditions as in \cref{lemma:cramer}, the following holds:
    \begin{itemize}
        \item For all $x>\mu$, $\bLp\PP(U^{(\tau(\epsilon))} \ge x)\bRp_{\epsilon>0}$ is almost regular with resistance $\tau_0\Lambda^*_U(x)$.
        \item For all $x<\mu$, $\bLp\PP(U^{(\tau(\epsilon))} < x)\bRp_{\epsilon>0}$ is almost regular with resistance $\tau_0\Lambda^*_U(x)$.
    \end{itemize}
\end{proposition}

\begin{proof}
    Let us show that if $(X^\tau)_\tau$ satisfies
    \begin{align}
        \label{eq:proof:prop:almost_regular_deviation}
        \lim_{\tau\to\pinfty}\frac{1}{\tau}\log\big(X^\tau\big) = -r
    \end{align}
    for some $r\ge0$, then $(Y^\epsilon)_{\epsilon>0} \eqdef \bLp X^{\tau(\epsilon)} \bRp_{\epsilon>0}$ is almost regular with resistance $\tau_0r$ as defined in \cref{def:almost_regularity}.
    Indeed, if \cref{eq:proof:prop:almost_regular_deviation} holds then for all $r'\ge0$
    \begin{align*}
        \frac{1}{\tau(\epsilon)}\log\left(e^{\tau(\epsilon)r'}X^{\tau(\epsilon)}\right)
        = r' + \frac{1}{\tau(\epsilon)}\log\left(X^{\tau(\epsilon)}\right)
        \underset{\epsilon\to0}{\to} r'-r
    \end{align*}
    as $\lim_{\epsilon\to0}\tau(\epsilon)\to\pinfty$.
    From there,
    \begin{itemize}
        \item If $r'<r$, $\lim_{\epsilon\to0}\log\big(e^{\tau(\epsilon)r'}X^{\tau(\epsilon)}\big) = \minfty$, hence $\lim_{\epsilon\to0}e^{\tau(\epsilon)r'} Y^\epsilon = 0$.
        \item If $r'>r$, $\lim_{\epsilon\to0}\log\big(e^{\tau(\epsilon)r'}X^{\tau(\epsilon)}\big) = \pinfty$, hence $\lim_{\epsilon\to0}e^{\tau(\epsilon)r'} Y^\epsilon = \pinfty$.
    \end{itemize}
    Now, since $\tau(\epsilon) = \lceil\tau_0 \log(\frac1\epsilon)\rceil$ is bounded between $\tau_0 \log(\frac1\epsilon)$ and $\tau_0 \log(\frac1\epsilon) + 1$, we have
    \[e^{-r'} e^{\tau(\epsilon) r'} \le \epsilon^{-\tau_0r'} \le e^{\tau(\epsilon) r'}\]
    and replacing $e^{\tau(\epsilon) r'}$ with $\epsilon^{-\tau_0r'}$ in the above statements yields the same limits.
    We conclude that $(Y^\epsilon)_{\epsilon>0}$ is almost regular with resistance~$\tau_0r$.
    Replacing $X^\tau$ with $\PP(U^{(\tau)} \ge x)$ or $\PP(U^{(\tau)} < x)$ and using \cref{lemma:cramer} concludes the proof. 
\end{proof}

Recall that we are ultimately interested in $\PP(U^{(\tau(\epsilon))} \in v)$ for any bin $v$. This is deduced directly from the above using \cref{lemma:Lambda_props}.
\begin{corollary}
    \label{corol:almost_regular_deviation}
    Let $\tau(\epsilon) = \lceil\tau_0 \log(\frac1\epsilon)\rceil$ for some constant $\tau_0>0$. Then for all random distribution $U$ supported within $[0,1]$ and bin $v$, $\PP(U^{(\tau(\epsilon))} \in v)$ is almost regular.
    Moreover, denoting $r_U(v)$ the corresponding resistance, we have
    \begin{align*}
        r_U(v) =
        \begin{cases}
            \tau_0\Lambda^*_U(v^-) & \text{if $\mu < v^-$,}\\
            0 & \text{if $\mu \in v$,}\\
            \tau_0\Lambda^*_U(v^+) & \text{if $\mu \ge v^+$.}
        \end{cases}
    \end{align*}
    Moreover, $\PP(U^{(\tau(\epsilon))} \in v)$ converges to a positive limit as $\epsilon\to0$ if $r_U(v) = 0$ and is equal to $0$ if $r_U(v) = \pinfty$.
    If $U$ is deterministic, the result remains true except that when $\mu=v^+$, $r_U(v)$ is equal to $\pinfty$ instead of $\tau_0\Lambda^*_U(\mu)=0$.
\end{corollary}

\begin{proof}
    If $U$ is deterministic, everything follows from the fact that $\PP(U^{(\tau(\epsilon))} \in v) = \charac{v}(\mu)$ and $\Lambda^*_U(x) = \pinfty$ for $x\ne\mu$.
    Assume that $U$ is not deterministic.
    If $v^-\le\mu\le v^+$, we already know that this probability converges to a positive limit as $\epsilon\to0$ (precisely, to $1$ if the inequalities are strict, else to $\frac12$), which implies that $\PP(U^{(\tau(\epsilon))} \in v)$ is almost regular with resistance~$0$. This concludes the case where $\mu\in v$ and the case where $\mu=v^+$ for which $\Lambda^*_U(v^+)=0$.
    If $\mu<v^-$, one can write \[\PP(U^{(\tau(\epsilon))} \in v) = \PP(U^{(\tau(\epsilon))} \ge v^-) - \PP(U^{(\tau(\epsilon))} \ge v^+)\] and both terms are almost regular with resistance $\tau_0\Lambda^*_U(v^-)$ and $\tau_0\Lambda^*_U(v^+)$ respectively,  according to \cref{prop:almost_regular_deviation}.
    Since $\Lambda^*_U$ is increasing over $[\mu,\pinfty)$, the second term is dominated by the first one, therefore $\PP(U^{(\tau(\epsilon))} \in v)$ is equivalent to $\PP(U^{(\tau(\epsilon))} \ge v^-)$, hence almost regular with resistance $\tau_0\Lambda^*_U(v^-) > 0$.
    If $\mu>v^+$, \[\PP(U^{(\tau(\epsilon))} \in v) = \PP(U^{(\tau(\epsilon))} < v^+) - \PP(U^{(\tau(\epsilon))} < v^-)\] is almost regular with resistance $\tau_0\Lambda^*_U(v^+) > 0$ using the same reasoning.
    We have identified the value of $r_U(v)$ for all cases and proven that it is equal to $0$ if and only if $\PP(U^{(\tau(\epsilon))} \in v)$ converges to a positive limit as $\epsilon\to0$.
    Moreover, the fact that $r_U(v) = \pinfty$ if and only if $\PP(U^{(\tau(\epsilon))} \in v) = 0$ stems from the second part of \cref{lemma:Lambda_props}.
\end{proof}

Beyond the theoretical requirement of regularity, it is also needed to lower bound the resistance of offset observations in order to ensure the stability of aligned states.
For common distributions, the \ac{lmgf} may be explicitly computed. For instance, given a gaussian variable $Z \sim \cN(\mu,\sigma^2)$, it holds that for $x>\mu$,
\begin{align*}
    \PP(Z^{(\tau)} > x)
    \underset{\tau\to\pinfty}{\sim} \frac{\sigma}{\sqrt{2\pi\tau}(x-\mu)} e^{-\tau\frac{(x-\mu)^2}{2\sigma^2}}
\end{align*}
hence $\Lambda^*_Z(v^-) = \frac{(x-\mu)^2}{2\sigma^2}$.
Note that this example shows that the introduction of almost regularity is necessary. Indeed, here $\PP(Z^{(\tau)} > x)$ is of order $\tau^{-1/2} e^{-\tau r}$ whereas regularity as defined in \cref{def:regularity} would impose an equivalent of order $e^{-\tau r}$.
In general, knowing that utilities are bounded \ac{as} within $[0,1]$, Hoeffding's inequality implies that for all~$x$
\begin{align}
    \label{eq:bound_Lambda}
    \Lambda^*_U(x) \ge 2(x-\mu)^2~.
\end{align}
If it is known that the variance of utilities is upper bounded by $\sigma^2>0$, then using Bernstein's inequality instead yields
\[\Lambda^*_U(x) \ge \frac{(x-\mu)^2}{2\sigma^2+2|x-\mu|}~,\]
which may be tighter when $\sigma$ is known to be small.
Precise knowledge of the utility distributions or more specific inequalities may improve the above lower bounds. This discussion falls under the theory of large deviations and is independent from our work.
Ultimately, \cref{eq:bound_Lambda} allows to lower bound the resistance of offset observations in the most general case and shall be used in the following.
Eventually, this inequality translates to a condition between $\tau_0$ and $\delta$ to ensure that such events are negligible in the \ac{ritel} process, see \cref{cond:R_0}.

\subsection{Almost Regular Perturbed Markov Processes}
\label{sec:ritel:arpmp}

In this section we show that the theory of \acp{rpmp} can be extended by replacing the regularity condition with almost regularity.
\Cref{thm:sss_min_gamma_AR} generalizes \cref{thm:sss_min_gamma} to this larger class of processes.

\begin{definition}[Almost Regular Perturbed Markov Process]
    \label{def:arpmp}
    Let $(P^\epsilon)$ be a \ac{pmp} over a state space $\cX$. $(P^\epsilon)$ is said to be an \ac{arpmp} if all transition probabilities are almost regular, that is
    \begin{align*}
        & \text{$\forall x,y\in\cX$, $\exists r\ge0:$}
        \begin{cases}
            \lim_{\epsilon\to0}\epsilon^{-r'}P_{x,y}^\epsilon = 0 & \text{for all $r'<r$,}\\
            \lim_{\epsilon\to0}\epsilon^{-r'}P_{x,y}^\epsilon = \pinfty & \text{for all $r'>r$.}
        \end{cases}
    \end{align*}
\end{definition}

As in \cref{def:regularity}, $r$ is unique if defined, is called the resistance of the transition and can be equal to $\pinfty$. This is the case in particular for non-existing transitions $P^\epsilon_{x,y}=0$.
Recall that \cref{eq:def:regularity} is stronger then \cref{eq:def:almost_regularity} and both notions of resistance coincide if regularity holds.
We define the notions of resistance $r(x\to y)$, resistance graph $\cG$, rooted tree, and potential $\gamma$ exactly as in \cref{sec:rpmp}.

When the process is almost regular we lose the information on the behavior of $\epsilon^{-r}P^\epsilon_{x,y}$. In particular we are unable to know how two transitions $P^\epsilon_{x,y}$ and $P^\epsilon_{x',y'}$ of equal resistance behave relatively, as their ratio can be any quantity that is sub-polynomial with regards to $\epsilon$.
This leads to not being able to describe as accurately the behavior of $\pi^\epsilon_x$ for states $x$ minimizing the potential as $\epsilon\to0$.
We are however able to adapt \cref{thm:sss_min_gamma} to show that states that do not minimize the potential vanish even under the looser assumption of almost regularity, so that the \acp{sss} are included in $\cX^\star$ although the inclusion may not be an equality.
\begin{theorem}
    \label{thm:sss_min_gamma_AR}
    Let $(P^\epsilon)$ be an \ac{arpmp} over a finite space $\cX$ such that $P^\epsilon$ is aperiodic and irreducible for all $\epsilon>0$ and denote $\pi^\epsilon$ its stationary distribution.
    Then, $\pi^\epsilon_x$ vanishes as $\epsilon\to0$ for any state that does not minimize $\gamma$.
    In particular, the set of \acp{sss} satisfies
    \[\cS
   \subset \cX^\star
    \eqdef \bLa x\in\cX~:~\gamma(x) = \min_\cX\gamma\bRa~.\]
\end{theorem}

As for \cref{thm:sss_min_gamma}, the potential can be computed over recurrence classes instead of the whole state space.
There is a slight loss of accuracy in \cref{thm:sss_min_gamma_AR} compared to \cref{thm:sss_min_gamma} as the behavior of $\pi^\epsilon_x$ for an individual state $x\in\cX^\star$ is unknown, whereas it it shown to converge to a positive limit if the process is a \ac{rpmp}.
This is not an issue in our context as we mainly care about the fact that the process asymptotically remains in favorable states, regardless of the distribution within these states.

\begin{proof}
    The proof of \cref{thm:sss_min_gamma} \cite[Lemma 1]{young93} states that $\pi^\epsilon$ can be expressed as
    \begin{align}
        \label{eq:thm:sss_min_gamma_AR:pi}
        \pi^\epsilon_x = \frac{p^\epsilon_x}{\sum_{y\in \cX}p^\epsilon_y}
    \end{align}
    for all $x\in\cX$, where
    \begin{align*}
        p^\epsilon_x \eqdef \sum_{T:x\text{-tree}} \prod_{(y\to z)\in T} P^\epsilon_{y,z}~.
    \end{align*}
    Proving \cref{thm:sss_min_gamma} consists in showing that for all $x\in\cX$, $\lim_{\epsilon\to0}\epsilon^{-\gamma(x)}p^\epsilon_x \in (0,\pinfty)$, \ie, $p^\epsilon_x$ is regular with resistance $\gamma(x)$.
    It follows that the quantities $\pi^\epsilon_x$ for all $x\in\cX^\star$ are all equivalent up to a constant as $\epsilon\to0$ and that $\pi^\epsilon_y$ for all $y\notin\cX^\star$ are negligible, hence the \acp{sss} are the states minimizing $\gamma$.
    When the process is almost regular instead, we may only show that $p^\epsilon_x$ is almost regular with resistance $\gamma(x)$. For all $x\in\cX$,
    \begin{align*}
        \begin{cases}
            \lim_{\epsilon\to0}\epsilon^{-\gamma'}P_{x,y}^\epsilon = 0 & \text{for all $\gamma'<\gamma(x)$,}\\
            \lim_{\epsilon\to0}\epsilon^{-\gamma'}P_{x,y}^\epsilon = \pinfty & \text{for all $\gamma'>\gamma(x)$.}
        \end{cases}
    \end{align*}
    Indeed, let $x\in\cX$ and $\gamma'<\gamma(x)$. Since $\gamma' < \gamma(x) \le \sum_{(y\to z)\in T} r(y\to z)$ for any $x$-tree $T$, there exists a constant $\alpha>0$ such that $\gamma' = \sum_{(y\to z)\in T} (r(y\to z)-\alpha)$.
    Then,
    \begin{align*}
        \epsilon^{-\gamma'} \prod_{(y\to z)\in T} P^\epsilon_{y,z}
        = \prod_{(y\to z)\in T} \epsilon^{-(r(y\to z)-\alpha)}P^\epsilon_{y,z}.
    \end{align*}
    Using \cref{eq:def:almost_regularity} with $r' = r(y\to z)-\alpha < r(y\to z)$, each term in this finite product goes to $0$ as $\epsilon\to0$. Summing over all $x$-tree, of which there is a finite amount, we conclude that $\lim_{\epsilon\to0} \epsilon^{-\gamma'}p^\epsilon_x = 0$.
    Let $\gamma'>\gamma(x)$ and $T$ be a minimal $x$-tree. Since $\gamma' > \gamma(x) = \sum_{(y\to z)\in T} r(y\to z)$, the exists a constant $\alpha>0$ such that $\gamma' = \sum_{(y\to z)\in T} (r(y\to z)+\alpha)$.
    Then,
    \begin{align*}
        \epsilon^{-\gamma'}p^\epsilon_x
        \ge \epsilon^{-\gamma'} \prod_{(y\to z)\in T} P^\epsilon_{y,z}
        = \prod_{(y\to z)\in T} \epsilon^{-(r(y\to z)+\alpha)}P^\epsilon_{y,z}.
    \end{align*}
    Using \cref{eq:def:almost_regularity} with $r' = r(y\to z)+\alpha > r(y\to z)$, each term in this finite product goes to $\pinfty$ as $\epsilon\to0$. We conclude that $\lim_{\epsilon\to0} \epsilon^{-\gamma'}p^\epsilon_x = \pinfty$.
    We have shown that $p^\epsilon_x$ is almost regular with resistance $\gamma(x)$, and can now rewrite \cref{eq:thm:sss_min_gamma_AR:pi} as
    \[\pi^\epsilon_x = \frac{\epsilon^{-\gamma'}p^\epsilon_x}{\sum_{y\in\cX^\star}\epsilon^{-\gamma'}p^\epsilon_{y} + \sum_{y\notin\cX^\star}\epsilon^{-\gamma'}p^\epsilon_y}\]
    for any $\gamma'\ge0$.
    If $x\notin\cX^\star$, choosing $\gamma'$ so that $\min\gamma < \gamma' < \gamma(x)$ makes the numerator go to $0$ and the denominator go to $\pinfty$, which implies that $\lim_{\epsilon\to0}\pi^\epsilon_x = 0$. This holds for all $x\notin\cX^\star$, therefore $\lim_{\epsilon\to0}\pi^\epsilon_{\cX^\star} = 1$.
\end{proof}

\subsection{Theoretical Analysis of RITEL}
\label{sec:ritel:results}

The previous sections introduce \ac{arpmp} processes and show that \ac{ritel} adheres to this framework.
We now discuss the convergence properties of \ac{ritel}. This section follows the same outline as \cref{sec:itel:results}.
We call state a triplet $(\blm,\bla,\blv) = (\lm_i,\la_i,\lv_i)_{i\in I}$ describing the states of all agents at a given time in the algorithm, and $\cX$ the set of all possible states.
As for \ac{itel}, we disregard the benchmark action and bin of discontent agents. We also identify the set of all-discontent states to a single state $D$.
Recall that our goal is to identify the \acp{sss} $\cS$, which are included in the states $\cX^\star$ that minimize the potential $\gamma$ according to \cref{thm:sss_min_gamma_AR} in the case where the perturbed process $P^\epsilon$ is aperiodic and irreducible.

Recall that \ac{itel} is shown to converge under the interdependence assumption \cref{assum:interdependence}.
In the context of \ac{ritel}, due to bin quantization and to the fact that offsets of only one bin are dismissed by agents, we require a slightly stronger assumption, that interdependence can always happen with a change of at least $3\delta$ in utility.
\begin{assumption}[$3\delta$-Interdependence]
\label{assum:3dinterdependence}
The game is $3\delta$-\emph{interdependent}, that is, given any action profile $\ba$ and any proper subset $\emptyset\subsetneq J\subsetneq I$ of agents, there exists an agent $i\notin J$ and a choice of actions $a'_J$ such that $|M_j(a'_i,a_{-i})-M_j(a)| \ge 3\delta$.
\end{assumption}
Note that since there is a finite number of agents and action profiles, a game that satisfies interdependence \cref{assum:interdependence} always satisfies $3\delta$-interdependence \cref{assum:3dinterdependence} for small enough $\delta$.

\subsubsection{Definitions}

Recall that when studying \ac{itel} in \cref{sec:itel:results}, we defined subsets $\cA\subset\cE\subset\cC\subset\cX$ of the state space $\cX$, namely $\cC$ the subset where all agents are content with a benchmark utility aligned with the benchmark action profile, $\cE$ the subset of equilibrium states, and $\cA$ the subset of optimal states.
With each inclusion comes a stronger sense of stability: states in $\cC$ are absorbing in the unperturbed process $P^0$, states in $\cE$ have a large outward resistance in the perturbed process $P^\epsilon$ and states in $\cA$ are absorbing in $P^\epsilon$.
Regarding \ac{ritel}, similar subsets are defined. However, their definition needs to be adapted to account for the loss of accuracy due to the introduction of random utilities and bin quantization.
Eventually, these subsets come with a ``weak'' and ``strong'' variant. Intuitively, the ``strong'' set contains the states that are ensured to hold the associated property, whereas the ``weak'' set contains the states that may hold the property and is identified with a subscript $\delta$.
We derive upper bounds on the potential of states that are part of the strong subset, and lower bound on the potential of the states that are not part of the weak subset. The weak set acts as a gray area that cannot be described as accurately.

Let us adapt \cref{def:aligned} to introduce the notion of strongly aligned states $\cC$ and weakly aligned states $\cC_\delta$ in \ac{ritel}. 
Since benchmark utilities are quantified into bins, one would intuitively define aligned states as states where the mean utility of each agent is within its benchmark bin. Such states shall be referred to as strongly aligned.
Recall that \ac{ritel} is designed so that agents observing an offset of only one bin from their benchmark disregard the change. For this reason we also introduce a weaker notion of alignment for agents with an offset of at most one bin between their benchmark and their mean.
\begin{definition}[Aligned States]
    \label{def:weak&strong_aligned}
    \hfill
    \begin{itemize}
        \item An agent with benchmark bin $\lv$ is \emph{strongly aligned} with an action profile $\ba$ if its average utility $\mu$ in $\ba$ satisfies $\mu\in\lv$.
        \item An agent with benchmark bin $\lv$ is \emph{weakly aligned} with an action profile $\ba$ if its average utility $\mu$ in $\ba$ satisfies $\mu\in\lv\pm\delta$. If $U(\ba)$ is not deterministic, it is also required that $\mu\ne(\lv-\delta)^-$.
    \end{itemize}
    An all-content state is strongly aligned (resp. weakly aligned) if all agents are strongly aligned (resp. weakly aligned) with the benchmark actions. 
    The set of weakly aligned states is denoted $\cC_\delta \subset \cX$ and the set of strongly aligned states is denoted $\cC \subset \cC_\delta$.
\end{definition}

Generally speaking, an agent is strongly aligned if its mean utility is contained in its benchmark bin, and weakly aligned if it is contained in its benchmark bin or in an adjacent bin.
There is one exception which is the case where the average utility $\mu$ lies exactly at the boundary between two bins $\nu$ and $\nu-\delta$ and the utility is not deterministic. In this case, $\nu-\delta$ is weakly aligned, $\nu$ is strongly aligned, but $\nu+\delta$ is not weakly aligned. This is due to the fact that as $\tau\to\pinfty$, $U^{(\tau)}$ will tend to be observed in both $\nu$ and $\nu-\delta$ with equal probability, and the resistance for observing $U^{(\tau)}$ in $v+\tau$ can be explicitly lower bounded since it requires an offset of at least $\delta$ from the true mean.
This technical detail does not influence the overall reasoning.
In any case, given an action profile $\ba$, there is exactly one strongly aligned state with benchmark actions $\ba$.
We shall see that states in $\cC_\delta$ are absorbing in $P^0$, but their outward resistance may be arbitrarily close to $0$, whereas the outward resistance of states in $\cC$ can be lower bounded.

We now adapt the notion of equilibrium given in \cref{def:equilibrium}.
In \ac{itel}, the notion of equilibrium for aligned states coincides with the notion of equilibrium for their benchmark action profile.
In \ac{ritel}, we need to take into account quantization and weakly aligned states. In particular, the equilibrium inequality itself needs to be relaxed by allowing small improvements.
We first define this notion for action profiles, then for weakly aligned states.
\begin{definition}[$\rho$-Equilibrium Action Profile]
    \label{def:equilibrium_delta_action}
    An action profile $\ba$ with resulting average utilities $\bmu$ is a $\rho$-equilibrium if for all agent $i$, $i$ is at a $\rho$-equilibrium, \ie, for any action $a'_i\ne a_i$, $M_i(a'_i,a_{-i}) \le \mu_i+\rho$.
\end{definition}
An action profile is at a $\rho$-equilibrium if no agent can explore such that it would improve its mean utility by at least $\rho$. When $\rho=0$, this notion coincides with the standard notion of equilibrium.
Regarding weakly aligned states, notice that for a given action profile $\ba$ there are several weakly aligned states with benchmark $\bla=\ba$ but with different benchmark bins $\blv$---although only one strongly aligned state. How stable the state is depends on $\blv$ since higher benchmarks are harder to leave, so the definition of equilibrium takes both $\bla$ and $\blv$ into account.
\begin{definition}[$\rho$-Equilibrium State]
    \label{def:equilibrium_delta_state}
    A weakly aligned state with benchmark actions $\bla$ and utility bins $\blv$ is a $\rho$-equilibrium if for all agent $i$, $i$ is at a $\rho$-equilibrium, \ie, for any action $a_i\ne \la_i$, $N_i(a_i,\la_{-i}) \le \lv_i+\rho$.
    We denote $\cE_\delta \subset \cC_\delta$ the set of weakly aligned $\delta$-equilibrium states and $\cE \subset \cE_\delta\cap\cC$ the set of strongly aligned equilibrium states.
\end{definition}
When $\rho=\delta$, a state is at a $\delta$-equilibrium if no agent can explore in a way so that it would observe a mean utility high enough that it can accept it, hence the notion of $\delta$-equilibrium state reflects how stable a state it is in \ac{ritel}.
On the other hand, the notion of $\rho$-equilibrium for action profiles is completely unrelated to the \ac{ritel} algorithm and instead reflects the intrinsic equilibrium properties of $\ba$ in the underlying game.
This notion is only used to translate the convergence statement of \ac{ritel} in terms of algorithm states to a more concrete statement involving action profiles and their mean utilities instead of benchmark bins.
Although closely related, both notions of equilibrium do not correspond directly, that is, a $\rho$-equilibrium state does not necessarily imply a $\rho$-equilibrium action profile. In fact, \cref{def:equilibrium_delta_action,def:equilibrium_delta_state} are designed so that an equilibrium action profile translates well into an equilibrium state, whereas an equilibrium state translates into a weaker sense of equilibrium action profile, see \cref{lemma:ritel:correspondence_equilibrium}.

It remains to discuss optimal states.
A weakly aligned agent is said to be optimized if its benchmark utility $\lv$ is optimal, \ie, $\{1\}$, and $\delta$-optimized if $\lv$ is $\delta$-optimal, \ie, either $[1-\delta,1)$ or $\{1\}$.
We denote $\cA_\delta \subset \cE_\delta$ the set of weakly aligned optimal states, where all agents are optimized.
The inclusion is due to the fact that if an agent is $\delta$-optimal, then it cannot improve its utility by more than $\delta$.
A ``strong'' set $\cA$ shall be defined later to identify the absorbing states of the process.

The following lemma states the correspondences that are required to convert the convergence result of \ac{ritel} in terms of states and benchmark bins to a convergence result in terms of action profiles and mean utilities.
\begin{lemma}
    \label{lemma:ritel:correspondence_equilibrium}
    The following holds.
    \begin{enumerate}[(i)]
        \item If $\ba$ is an equilibrium action profile, then the corresponding strongly aligned state $x$ is in $\cE$.
        \item If $x\in\cA_\delta$, then its benchmark action profile $\bla$ is $\delta$-optimal.
        \item If $x\in\cE_\delta$, then its benchmark action profile $\bla$ is a $3\delta$-equilibrium.
    \end{enumerate}
\end{lemma}

\begin{proof}
    \hfill
    \newline\noindent\textbf{(i).}
    Let $\ba$ be an equilibrium action profile of mean utilities $\bmu$ with bins $\bnu$, and $x=(\blm,\bla,\blv)\in\cC$ with $\bla=\ba$ and $\blv=\bnu$. For all agent $i$ and action $a'_i\ne a_i=\la_i$,
    $$N_i(a'_i,\la_{-i}) = |M_i(a'_i,a_{-i})| \le |\mu_i| = \lv_i$$
    since $M_i(a'_i,a_{-i})) \le \mu_i$ by assumption on $\ba$, hence $i$ is at an equilibrium in $x$.
    Therefore $x$ is an equilibrium state.

    \medskip\noindent\textbf{(ii).}
    Let $x=(\blm,\bla,\blv)\in\cA_\delta$ and denote $\bmu$ the mean utilities of $\bla$.
    For all agent $i$, since $i$ is optimized in $x$ and $x$ is weakly aligned, $\mu_i \ge \lv_i^--\delta = 1-\delta$, hence $i$ is $\delta$-optimized by $\bla$.
    Therefore $\bla$ is $\delta$-optimal.
    
    \medskip\noindent\textbf{(iii).}
    Let $x=(\blm,\bla,\blv)\in\cE_\delta$ and denote $\bmu$ the mean utilities of $\bla$.
    For all agent $i$ and action $a'_i\ne \la_i$, since $i$ is at a $\delta$-equilibrium in $x$ and $x$ is weakly aligned, $M_i(a'_i,\la_{-i}) \le N_i(a'_i,\la_{-i})^-+\delta \le \lv_i^-+2\delta \le \mu_i+3\delta$, hence $i$ is at a $3\delta$-equilibrium in $\bla$.
    Therefore $\bla$ is a $3\delta$-equilibrium.
\end{proof}

We now move on to the notions of welfare and stability, along with their virtual counterparts, defined in \cref{def:itel:W&S,def:itel:tW&tS}. As for the notion of equilibrium, virtual welfare and stability are defined both in the context of states and in the context of action profiles. True welfare and stability is however used only to state the final result in terms of action profiles, in the special case of linear functions $F$ and $G$ for simplicity, hence they are only defined in the context of action profiles.
\begin{definition}[Welfare and Stability of an Action Profile]
    \label{def:ritel:W&S_action}
    Let $\ba$ be an action profile with mean utilities~$\bmu$. We define
    \begin{itemize}
        \item $W(\ba) \eqdef \sum_i\mu_i$,
        \item $S(\ba) \eqdef \max \bLa M_i(a'_i,a_{-i})-\mu_i ~|~ \textrm{$i\in I$ and $a'_i\ne a_i$ such that $M_i(a'_i,a_{-i}) > \mu_i$}\bRa$.
    \end{itemize}
    $S(\ba)$ is defined only if $\ba$ is not an equilibrium.
\end{definition}

\begin{definition}[Virtual Welfare and Stability of an Action Profile]
    \label{def:ritel:tW&tS_action}
    Let $\ba$ be an action profile with mean utilities~$\bmu$. We define
    \begin{itemize}
        \item $\tW_+(\ba) \eqdef 1 - \sum_iF(\mu_i+\delta)$,
        \item $\tW_-(\ba) \eqdef 1 - \sum_iF(\mu_i-\delta)$,
        \item $\tS_+(\ba) \eqdef 1 - \min \bLa G\bLp\mu_i-\delta, M_i(a'_i,a_{-i})+\delta\bRp ~|~ \textrm{$i\in I$ and $a'_i\ne a_i$ such that $M_i(a'_i,a_{-i}) > \mu_i$}\bRa$,
        \item $\tS_-(\ba) \eqdef 1 - \min \bLa G\bLp\mu_i+\delta, M_i(a'_i,a_{-i})-\delta\bRp ~|~ \textrm{$i\in I$ and $a'_i\ne a_i$ such that $M_i(a'_i,a_{-i}) > \mu_i+3\delta$}\bRa$.
    \end{itemize}
    $\tS_+(\ba)$ is defined only if $x$ is not a $\delta$-equilibrium and $\tS_-(\ba)$ only if $x$ is not a $3\delta$-equilibrium.
\end{definition}

The subscript $+$ (resp. $-$) indicates that the quantity defined acts as an upper bound (resp. lower bound) on the virtual welfare or stability. The introduction of these offset definitions are motivated by the quantization of \ac{ritel}. For instance, for a state $x=(\blm,\bla,\blv)\in\cC$ there may be an offset of $\delta$ at most between the utilities $\blv^-$ used for the functions $F$ and $G$ and the mean utilities resulting from the benchmark actions $\bla$.
The link between true welfare and stability and these quantities is discussed in the linear case at the end of \cref{sec:ritel:results:results}.

Regarding the virtual welfare and stability of states, the randomness of the observation makes the analysis of acceptations from a content state more complex. Indeed, an agent of benchmark $\lv$ may accept the mean bin $\nu$ of its exploration with resistance equal to $G(\lv,\nu)$ when $\nu \ge \lv + 2\delta$. It may also accept any higher bin $v$ with resistance equal to $r_U(v) + G(\lv,v)$ where $U$ is the distribution of the observed utility. This resistance may be lower than that of accepting $\nu$ if the non-increasing nature of $G(\lv,\cdot)$ compensate for the addition of $r_U(v)$.
For this reason, we also introduce two quantities $\tS_+$ and $\tS_-$ that act as upper and lower bounds on the virtual stability of a state.
\begin{definition}[Virtual Welfare and Stability of a State]
    \label{def:ritel:tW&tS_state}
    Let $x=(\blm,\bla,\blv)\in\cC_\delta$. We define
    \begin{itemize}
        \item $\tW(x) \eqdef 1 - \sum_i F(\lv_i)$,
        \item $\tS_+(x) \eqdef 1 - \min \bLa G(\lv_i,N_i(a_i,\la_{-i})+\delta) ~|~ \textrm{$i\in I$ and $a_i\ne\la_i$ such that $N_i(a_i,\la_{-i}) \ge \lv_i+\delta$}\bRa$,
        \item $\tS_-(x) \eqdef 1 - \min \bLa G(\lv_i,N_i(a_i,\la_{-i})) ~|~ \textrm{$i\in I$ and $a_i\ne\la_i$ such that $N_i(a_i,\la_{-i}) \ge \lv_i+2\delta$}\bRa$.
    \end{itemize}
    $\tS_+(x)$ is defined only if $x$ is not an equilibrium and $\tS_-(x)$ only if $x$ is not a $\delta$-equilibrium.
\end{definition}

Notice that $\tS_-(x) \le \tS_+(x)$ when both are defined, \ie, when $x$ is not a $\delta$-equilibrium.
The following lemma states the correspondences that are required to convert the convergence result of \ac{ritel} in terms of states and their virtual welfare and stability to a convergence result in terms of action profiles and their welfare and stability.
\begin{lemma}
    \label{lemma:ritel:correspondence_tW&tS}
    Let $x$ be a state with benchmark actions $\bla$. The following holds.
    \begin{enumerate}[(i)]
        \item If $x\in\cC_\delta$, $\tW(x) \le \tW_+(\bla)$.
        \item If $x\in\cC$, $\tW(x) \ge \tW_-(\bla)$.
        \item If $x\in\cC_\delta\setminus\cE_\delta$, $\tS_-(x) \ge \tS_-(\bla)$.
        \item If $x\in\cC\setminus\cE$, $\tS_+(x) \le \tS_+(\bla)$.
    \end{enumerate}
\end{lemma}

\begin{proof}
    Let $x=(\blm,\bla,\blv)\in\cC_\delta$ and denote $\bmu$ the mean utilities of $\bla$.
    Recall that $F$ and $G$ are defined over bins by replacing the bins with their lower edge: $F(v)=F(v^-)$ and $G(\lv,v)=G(\lv^-,v^-)$.
    
    \medskip\noindent\textbf{(i)}
    If $x$ is weakly aligned, then $\mu_i \ge \lv_i^--\delta$ for each agent $i$. $F$ is non-increasing, hence
    \[1-\tW(x)
    = \sum_i F(\lv_i^-)
    \ge \sum_i F(\mu_i+\delta)
    = 1-\tW_+(\bla)~.\]

    \medskip\noindent\textbf{(ii)}
    If $x$ is strongly aligned, then $\mu_i \le \lv_i^-+\delta$ for each agent $i$. $F$ is non-increasing, hence
    \[1-\tW(x)
    = \sum_i F(\lv_i^-)
    \le \sum_i F(\mu_i-\delta)
    = 1-\tW_-(\bla)~.\]

    \medskip\noindent\textbf{(iii)}
    If the minimum in the definition of $\tS_-(\bla)$ is over an empty set, the result reads $\tS(x) \ge \minfty$ which is immediate.
    Else, denote $a_i\ne\la_i$ the action that achieves the minimum. $M_i(a_i,\la_{-i}) > M_i(\bla)+3\delta$ so $N_i(a_i,\la_{-i}) \ge N_i(\bla)+3\delta \ge \lv_i+2\delta$ since $x$ is weakly aligned, which also implies that $M_i(\bla) \ge \lv_i^--\delta$. $G$ is non-decreasing in the first argument and non-increasing in the second, hence
    \[1-\tS_-(\bla)
    = G\bLp M_i(\bla)+\delta, M_i(a_i,\la_{-i})-\delta \bRp
    \ge G\bLp \lv_i^-, M_i(a_i,\la_{-i})-\delta \bRp
    \ge G\bLp \lv_i^-, N_i(a_i,\la_{-i})^- \bRp
    \ge 1-\tS_-(x)\]
    where the last inequality holds due to the fact that $N_i(a_i,\la_{-i}) \ge \lv_i+2\delta$.

    \medskip\noindent\textbf{(iv)}
    Denote $a_i\ne\la_i$ the action that achieves the minimum in the definition of $\tS_+(x)$ and $v_i=N_i(a_i,\la_{-i})$, so that $\tS_+(x) = 1-G(\lv_i^-,v_i^-+\delta)$. $x$ is strongly aligned so $M_i(\bla) < \lv_i^-+\delta$. $G$ is non-decreasing in the first argument and non-increasing in the second, hence
    \[1-\tS_+(x) = G\bLp \lv_i^-, v_i^-+\delta \bRp
    \ge G\bLp M_i(\bla)-\delta, v_i^-+\delta \bRp
    \ge G\bLp M_i(\bla)-\delta, M_i(a_i,\la_{-i})+\delta \bRp
    \ge 1-\tS_+(\bla)\]
    where the last inequality holds due to the fact that $v_i \ge \lv_i+\delta$, hence $M_i(a_i,\la_{-i}) \ge v_i^- \ge \lv_i^-+\delta > M_i(\bla)$.
\end{proof}

\subsubsection{Noise Resistance}

Let us discuss the resistance associated with an agent's decision. In \ac{itel} such resistances were derived directly from the policies given in \cref{tab:itel:policies}.
Here, the resistance of observing a utility enabling a given decision needs to be added to the resistance of the decision itself which is derived from \cref{tab:ritel:policies}.
In the example where a content agent who explored accepts its exploration with new benchmark bin $v$, recall that the probability of such decision given the chosen action profile is $\PP\bLp U^{(\tau)}\in v \bRp \epsilon^{G(\lv,v)}$. The associated resistance is $r_U(v) + G(\lv,v)$ where $r_U(v)$ is the resistance associated with $U^{(\tau)}$ belonging to $v$, defined in \cref{corol:almost_regular_deviation}.
In particular, recall that observing the mean bin has no resistance, whereas the resistance of observing other bins is lower bounded using \cref{eq:bound_Lambda}.
We are particularly interested in lower bounding the resistance of offset observation, which we recall corresponds to observing an average utility that is not weakly aligned with its theoretical mean.
\begin{lemma}
    \label{lemma:ritel:offset_observation}
    Let $U$ be a utility distribution with support in $[0,1]$ and mean bin $\nu$. Then for all $v\ne\nu\pm\delta$,
    \begin{align*}
        r_U(v) \ge R_0 \eqdef 2\tau_0\delta^2~.
    \end{align*}
\end{lemma}

\begin{proof}
    Let $U$ be a utility distribution, $\mu$ its mean utility and $\nu=|\mu|$ its mean bin.
    If $U$ is deterministic, the result is immediate since $\PP(U\in\nu) = 1$.
    Else, according to \cref{corol:almost_regular_deviation,eq:bound_Lambda}, the resistance associated with observing bin $v\ge\nu+2\delta$ is
    \begin{align*}
        r_U(v)
        = \tau_0\Lambda^*_U(v^-)
        \ge \tau_0\Lambda^*_U(\mu+\delta)
        \ge 2\tau_0(\mu+\delta-\mu)^2
        = 2\tau_0\delta^2 
    \end{align*}
    as $\Lambda^*_U$ is increasing over $[\mu,\pinfty)$ and $\mu+\delta \le \nu^++\delta = \nu^-+2\delta \le v^-$.
    The same reasoning holds when $v\le\nu-2\delta$ and then $r_U(v) \ge \tau_0\Lambda^*_U(\mu-\delta) \ge 2\tau_0\delta^2$.
\end{proof}

We now define the noise resistance of a state, which is the resistance associated with one agent observing a utility that is not aligned with its benchmark, in the event where the benchmark action profile is selected.
\begin{definition}[Noise Resistance]
    \label{def:ritel:noise_stability}
    Let $x=(\blm,\bla,\blv)\in\cX$. We define the \emph{noise resistance} of $x$ as
    \begin{align*}
        \tR(x)
        \eqdef \min \bLa r_{U_i(\bla)}(v) ~|~ v\ne\lv_i\pm\delta \bRa~.
    \end{align*}
\end{definition}

The minimum is considered over bins $v\ne\lv_i\pm\delta$ as these bins are the ones that imply a change of behavior if observed, according to \cref{tab:ritel:update_policy}.
This resistance is closely related to the alignment property. In order to identify the easy edges and outward resistance of states, we establish bounds on $\tR$ depending on whether or a state is aligned.
\begin{lemma}
    \label{lemma:ritel:noise_stability}
    Let $x=(\blm,\bla,\blv)\in\cX$. The following holds.
    \begin{enumerate}[(i)]
        \item If $x\notin\cC_\delta$, $\tR(x) = 0$.
        \item If $x\in\cC_\delta$, $\tR(x) > 0$.
        \item If $x\in\cC$, $\tR(x) \ge R_0$.
    \end{enumerate}
\end{lemma}

\begin{proof}
    \hfill\newline\noindent\textbf{(i).}
    If $x\notin\cC_\delta$, let $i$ be a non-aligned agent in $x$ with mean utility $\mu_i$, which bin $\nu_i$ satisfies $\nu_i\ne\lv_i\pm\delta$.
    According to \cref{corol:almost_regular_deviation}, $r_{U_i(\bla)}(\nu_i)=0$, hence $\tR(x)=0$.
    This covers all cases except for the one where $\mu_i=\nu_i^-=\lv_i^--\delta$ and $U_i(\bla)$ is not deterministic (see \cref{def:weak&strong_aligned}). In this case, we have $\nu_i-\delta=\lv_i-2\delta\ne\lv_i\pm\delta$ and \cref{corol:almost_regular_deviation} states that $r_{U_i(\bla)}(\nu_i-\delta)=0$ as $\mu_i=(\nu_i-\delta)^+$, hence $\tR(x)=0$.
     Note that this edge case was defined as not weakly aligned in \cref{def:weak&strong_aligned} precisely for this reason, so that the following property holds true for all weakly aligned states.

    \medskip\noindent\textbf{(ii).}
    If $x\in\cC_\delta$, then $\nu_i=\lv_i\pm\delta$ for all agents $i$. Then any bin $v\ne\lv_i\pm\delta$ satisfies $v\ne\nu_i$, hence $r_{U_i(\bla)}(v)>0$ according to \cref{corol:almost_regular_deviation}. Therefore $\tR(x)>0$.
    
    \medskip\noindent\textbf{(iii).}
    If $x\in\cC$, then $\nu_i=\lv_i$ for all agents $i$ and the fact that $\tR(x) \ge R_0$ follows directly from \cref{lemma:ritel:offset_observation}.
\end{proof}

The quantity $R_0$ may be seen as an indicator of how reliable the utilities sampled by \ac{ritel} are. Assuming $R_0$ is large enough, offset observations will play a negligible role in leaving strongly aligned states compared to other behaviors that were already studied in \ac{itel}.
A condition between $\delta$ and $\tau_0$ is introduced to lower bound the resistance of offset observations through $R_0$.
\begin{condition}[Noise control]
    \label{cond:R_0}
    \begin{align*}
        R_0
        = 2\tau_0\delta^2
        \ge 1.
    \end{align*}
\end{condition}

In general, the key condition needed is that $\tau_0\Lambda^*_U(\mu+\delta) > 1$ for all utility distributions $U$ with expectation $\mu$. This is ensured by \cref{cond:R_0} thanks to \cref{eq:bound_Lambda}.
As discussed in \cref{sec:ritel:almost_reg}, \cref{cond:R_0} may be weakened with better knowledge of the distributions involved. Again, this discussion is independent from our work.

In the \ac{ritel} process, we shall see that there exists a path leaving any state $x$ with resistance depending on $\tR(x)$ due to the randomness of observations.
\Cref{lemma:ritel:noise_stability,cond:R_0} are key to ensure that in the case of strongly aligned states, such paths are not easy edges, as they necessarily imply offset observations which have large resistance.
On the other hand, states that are not weakly aligned can be exited with no resistance.
Regarding weakly aligned states, they act as a gray area as their outward resistance is positive but may be arbitrarily close to $0$.
With the same reasoning as for \ac{itel}, this eventually implies that that the recurrence classes of \ac{ritel} are the weakly aligned states similarly to \cref{prop:itel:recurrence_P_0}.
\begin{proposition}
    \label{prop:ritel:recurrence_P_0}
    The recurrence classes of the unperturbed process $P^0$ are the singletons $\{x\}$ for each weakly aligned states $x\in\cC$, and possibly the communication class of the all-discontent state $D$.
\end{proposition}

\Cref{prop:ritel:recurrence_P_0} is proven in \cref{sec:ritel:proof_P_0}.
Let us now discuss the role of optimal states.
As in \ac{itel}, in an optimal state $x\in\cA_\delta$ no agent can explore. However, agents may make offset observations, so it is not true that all states in $\cA_\delta$ will be absorbing even in $P^\epsilon$, regardless of whether they are strongly aligned or not. In fact, we show that a state $x\in\cA_\delta$ is absorbing only if $\tR(x) = \pinfty$.
For this reason, we define $\cA \subset \cA_\delta$ the set of optimal states with infinite noise resistance $\tR$.
Although strongly aligned states are more likely to satisfy this property as suggested by \cref{lemma:ritel:noise_stability}, weakly aligned states may satisfy it too, for example when an utility is bounded \ac{as} in two consecutive bins, hence the agent cannot leave the weakly aligned bin with an offset observation.
\begin{proposition}
    \label{prop:ritel:recurrence_P_eps}
    Let $\epsilon>0$.
    \begin{itemize}
        \item If $\cA\ne\emptyset$, states $x\in\cA$ are absorbing states for the perturbed process $P^\epsilon$ and the recurrence classes of $P^\epsilon$ are exactly the corresponding singletons $\{x\}$.
        \item If $\cA=\emptyset$, $P^\epsilon$ is aperiodic and irreducible.
    \end{itemize}
\end{proposition}

\Cref{prop:ritel:recurrence_P_eps} is proven in \cref{sec:ritel:proof_res}.
As for \ac{itel}, in the first case we immediately conclude that the $\cS = \cA$, whereas in the second case we resort to using \cref{thm:sss_min_gamma_AR}.

\subsubsection{Convergence of RITEL}
\label{sec:ritel:results:results}

We now state our main convergence result for \ac{ritel}. Note that $F$ and $G$ must satisfy the same condition \cref{cond:itel} as \ac{itel}.
\begin{theorem}
    \label{thm:ritel:X_star}
    Assume that \cref{assum:3dinterdependence}, \cref{cond:itel}, and \cref{cond:R_0} hold and that $\cA=\emptyset$. Then the set $\cX^\star$ of states with minimum potential satisfies:
    \begin{itemize}
        \item If $\cE\ne\emptyset$, $\cX^\star \subset \cA_\delta \cup \bLa x\in\cE_\delta\setminus\cA_\delta ~:~ \tW(x) \ge \max_\cE \tW\bRa$.
        \item If $\cE=\emptyset$, $\cX^\star \subset \cE_\delta \cup \bLa x\in\cC_\delta\setminus\cE_\delta ~:~ \tW(x)-\tS_-(x) \ge \max_\cC (\tW-\tS_+)\bRa$.
    \end{itemize}
\end{theorem}

\Cref{thm:ritel:X_star} is proven in \cref{sec:ritel:proof}.
\Cref{prop:ritel:recurrence_P_0,prop:ritel:recurrence_P_eps} appear as key steps of the proof.
Recall that \cref{thm:sss_min_gamma_AR} then identifies the \ac{sss} to be included within $\cX^\star$. Then, \cref{lemma:ritel:correspondence_equilibrium,lemma:ritel:correspondence_tW&tS} allow to translate \cref{thm:ritel:X_star} to a statement in terms of action profiles.

\begin{theorem}[\ac{ritel} convergence]
\label{thm:ritel:convergence}
    Assume that \cref{assum:3dinterdependence}, \cref{cond:itel}, and \cref{cond:R_0} hold.
    \begin{itemize}
        \item If there exist equilibrium action profiles, the \ac{sss} in the \ac{ritel} process have action profiles $\ba$ that are either $\delta$-optimal, or $3\delta$-equilibria with $\tW_+(\ba) \ge \tW_-(\ba^\star)$ where $\ba^\star$ maximizes $\tW$ among equilibrium action profiles.
        \item Else, the \ac{sss} in the \ac{ritel} process have action profiles $\ba$ that either are $3\delta$-equilibria, or satisfy $\tW_+(\ba) - \tS_-(\ba) \ge \tW_-(\ba^\star) - \tS_+(\ba^\star)$ where $\ba^\star$ maximizes $\tW - \tS$ among all action profiles.
    \end{itemize}
\end{theorem}

The details needed to obtain \cref{thm:ritel:convergence} from the previous results are given in \cref{sec:ritel:convergence_in_action_profiles}.
Notice that since the game is finite, the set of all mean utilities involved is finite itself. 
In particular, assuming that $\delta$ is chosen small enough, all occurrences of $\delta$ may be replaced with $0$ in \cref{thm:ritel:convergence}:
\begin{itemize}
    \item If there exist equilibrium action profiles, the \ac{sss} in the \ac{ritel} process have equilibrium action profiles maximizing $\tW$ among equilibrium action profiles.
    \item Else, the \ac{sss} in the \ac{ritel} process have action profiles maximizing $\tW - \tS$ among all action profiles.
\end{itemize}
In other words, \ac{ritel} is capable of identifying optimal action profiles provided that $\delta$ is small enough. Such choice of $\delta$ would however imply a choice of $\tau_0$ that is unlikely to be reasonable in a practical setting in order to satisfy \cref{cond:R_0}.
In the case where $F$ and $G$ are chosen of the form $F:u\mapsto \phi_F - \psi_F\cdot u$ and $G:(\lu,u)\mapsto \phi_G - \psi_G\cdot(u-\lu)$, $\tW$ is a linear function of $W$ and $\tS$ is a linear function of $S$, which allows to translate the comparisons involving $\tW_\pm$ and $\tS_\pm$ to comparisons involving $W$ and $S$. \Cref{thm:ritel:convergence} then results in the following.
\begin{itemize}
    \item If there exist equilibrium action profiles, the \ac{sss} in the \ac{ritel} process have action profiles $\ba$ that are either $2\delta$-optimal, or $3\delta$-equilibria with $W(\ba) \ge \max W - 2\delta$ where the maximum is taken over equilibrium action profiles.
    \item Else, the \ac{sss} in the \ac{ritel} process have action profiles $\ba$ that either are $3\delta$-equilibria, or satisfy $\psi_FW(\ba)-\psi_GS(\ba) \ge \max \bLp\psi_FW-\psi_GS\bRp - 4\delta$ where the maximum is taken over all action profiles.
\end{itemize}

Indeed, in this case maximizing $\tW$ is equivalent to maximizing $W$ and maximizing $\tW-\tS$ is equivalent to maximizing $\psi_F\tW-\psi_G\tS$. Moreover, \cref{cond:itel} ensures that $n\psi_F+\psi_G \le 1$.
Recall that $\tW_\pm$ and $\tS_\pm$ are defined in \cref{def:ritel:tW&tS_action}.
The above statement is deduced from showing that
\begin{itemize}
    \item $\tW_+(\ba) \ge \tW_-(\ba')$ implies $W(\ba) \ge W(\ba')-2\delta$,
    \item $\tW_+(\ba)-\tS_-(\ba) \ge \tW_-(\ba')-\tS_+(\ba')$ implies $\psi_FW(\ba)-\psi_GS(\ba) \ge \psi_FW(\ba')-\psi_GS(\ba') - 4\delta$.
\end{itemize}
Indeed, denote $\bmu$ the mean utilities of $\ba$. Then
\[\tW_+(\ba)
= 1 - \sum_iF(\mu_i+\delta)
= 1 - n\phi_F + \psi_F\sum_i\mu_i + n\psi_F\delta
= 1 - n\phi_F + \psi_F (W(\ba)+n\delta)~.\]
Similarly, $\tW_-(\ba') = 1 - n\phi_F + \psi_F (W(\ba')-n\delta)$, which concludes the first implication as $2n\psi_F\delta \le 2\delta$. Now, Denote $\mu$ and $\mu'$ the mean utilities such that $S(\ba)=\mu'-\mu$. Then
\[\tS_-(\ba)
= 1 - G\bLp(\mu'-\delta)-(\mu+\delta)\bRp
= 1 - G(S(\ba)-\delta)
= 1 - \phi_G +\psi_G(S(\ba)-2\delta)~.\]
Similarly, $\tS_+(\ba') = 1 - \phi_G + \psi_G(S(\ba')+2\delta)$, which concludes the second implication as $2n\psi_F\delta + 4\psi_G\delta \le 4\delta$.

\section{Dynamic Perturbation}
\label{sec:gsa}

Recall that our analysis of \ac{itel} and \ac{ritel} discussed the nature of the \acp{sss}, that is the minimal subset of states $\cS$ such that
\begin{align*}
    \lim_{\epsilon\to0}~\liminf_{k\to\pinfty}~ \PP\bLp X_k^\epsilon\in\cS\bRp = 1~.
\end{align*}
In practice this means that for a fixed $\epsilon>0$ sufficiently small, the process $X^\epsilon$ eventually spends most of the time in states $x\in\cS$.
It is of interest to decrease the perturbation factor $\epsilon$ as the algorithm is running. A stronger result would be
\begin{align}
    \label{eq:single_limit}
    \lim_{k\to\pinfty}~ \PP\bLp X_k\in\cS\bRp = 1~,
\end{align}
where $(X_k)_{k\ge0}$ follows an inhomogeneous Markov chain $(P^{\epsilon_k})_{k\ge0}$ for some well chosen sequence $(\epsilon_k)_{k\ge0}$.
In fact, this kind of process can be interpreted as a form of \ac{gsa}.
In this section we apply \ac{gsa} results from \cite{trouvé1996} to improve \cref{thm:sss_min_gamma}.
We show that \cref{eq:single_limit} holds assuming $(\epsilon_k)_{k\ge0}$ is decreasing slowly enough. Moreover, given a finite horizon $H$, one can compute a ``cooling'' schedule $(\epsilon_k)_{0\le k\le H}$ with optimal convergence speed.
\ac{gsa} as defined in \cite[Definitions 1.1-1.3]{trouvé1996} can be described as follows.

\begin{definition}[Generalized Simulated Annealing]
    \label{def:gsa}
    Let $q$ be an aperiodic and irreducible Markov kernel over the state space $\cX$, $\kappa\in[1,\pinfty)$, and $V:\cX\times\cX\to[0,\pinfty]$ as cost function. A family of homogeneous Markov chains $(Q^T)_{0\le T\le T_0}$ over $\cX$ is \emph{admissible} for $(q,\kappa,V)$ if it satisfies the following properties for all $x,y\in\cX$:
    \begin{enumerate}[(i)]
        \item $Q_{x,y}^T \xrightarrow[T\to0]{} Q_{x,y}^0$,
        \item $V(x,y)<\pinfty$ if and only if $q_{x,y}>0$,
        \item for all $T\in[0,T_0]$, $\frac1\kappa q_{x,y}e^{-V(x,y)/T} \le Q^T_{x,y} \le \kappa q_{x,y}e^{-V(x,y)/T}$.
    \end{enumerate}
    When $(Q^T)$ is admissible, a \ac{gsap} is an inhomogeneous Markov process $(Q^{T_k})_{k\ge0}$ for some non-negative cooling schedule $(T_k)_{k\ge0}$.
\end{definition}

In \cite{trouvé1996} the first property is replaced with $(Q^T)_T$ being a continuous family, which is stronger. However only the continuity at $T=0$ is ever needed. Regardless, in the case of \ac{itel} transition probabilities are continuous functions of $\epsilon$.

In a \ac{gsap}, the transition from state $x$ to state $y$ behaves similarly to $e^{-V(x,y)/T}$ where $V\ge0$ is a cost function and $T>0$ is the temperature of the system. The parallel with \acp{pmp} is done by setting $\epsilon = e^{-1/T}$, $P^\epsilon=Q^T$ and $r(x\to y) = V(x,y)$, in which case the previous quantity is exactly $\epsilon^{r(x\to y)}$. Low temperature $T$ is then equivalent to small perturbation $\epsilon$.
Using this correspondence, we see that a \ac{rpmp} as in \cref{def:pmp} is also a \ac{gsa} as in \cref{def:gsa}. In fact, both definitions are equivalent except for the fact that the regularity condition is slightly weaker in \cref{def:gsa} than in \cref{def:regularity}.
It is however stronger then almost regularity as defined in \cref{def:almost_regularity}, hence we may not apply the theory of \ac{gsa} to \ac{ritel}.
It is possible however that some adaptation could be done to accommodate both frameworks.

The next section presents results from \cite{trouvé1996} translated into the \ac{rpmp} vocabulary. These results are stronger than \cref{thm:sss_min_gamma}, so that applying them in the context of \ac{itel} yields a more powerful statement than \cref{thm:itel:convergence}.

\subsection{GSA Theoretical Results}
\label{sec:eps_k:results}

In \cite{trouvé1996}, our notion of resistance appears via the cost $V$.
The notion of $x$-trees also appears as ``$\{x\}$-graphs'' \cite[Definition 1.4]{trouvé1996}, and potential appears as ``virtual energy'' \cite[Definition 1.5]{trouvé1996}.
The results from \cite{trouvé1996} are proven using a cycle decomposition. This decomposition is done by aggregating states from $\cX$ into ``cycles'' and deriving a new cost function over the graph of cycles, then iterating the process until the graph is reduced to a single state.
A key idea is that two states are merged into a cycle if there is a path going from one to the other and back using only easy edges.

The results are presented using quantities that depend on the cycle decomposition.
In the following, consider a \ac{rpmp} $(P^\epsilon)$ such that $P^\epsilon$ is aperiodic and irreducible for all $\epsilon>0$. Denote $\cX$ its state space and $\gamma$ its potential.
For $\lambda>0$, define $\cX_\lambda = \bLa x\in\cX: \gamma(x)\ge\gamma^\star+\lambda\bRa$ the set of states that are at least $\lambda$ sub-optimal, where $\gamma^\star=\min_{x\in\cX}\gamma(x)$.
For all cycle $\Pi$, denote $\gamma(\Pi) = \min_{x\in\Pi}\gamma(x)$.
The results below discuss the probability of the process to remain within $\cX_\lambda$.
Note that since there is a finite amount of state, when $\lambda$ is small enough $\cX_\lambda = \cX\setminus\cX^\star$.
Similarly, all quantities depending on $\lambda$ that are introduced in the following are constant at the neighborhood of $\lambda=0$, so that we eventually control $\PP(X_k\in\cX^\star)$ by applying the result to small enough $\lambda$.

\begin{theorem}[\cite{trouvé1996}, Theorem 5.2]
    \label{thm:trouvé_convergence}
    Let $(\epsilon_k)_{k\ge0}$ be a non-increasing cooling schedule.
    Let $(X_k)_{k\ge0}$ be an inhomogeneous Markov process following $(P^{\epsilon_k})_{k\ge0}$ with any starting distribution.
    Then for all $\lambda>0$,
    \begin{align*}
        \lim_{k\to0}~ \PP(X_k\in\cX_\lambda) = 0
        \quad \text{if and only if} \quad
        \sum_{k\ge0} \epsilon_k^{\Gamma_\lambda} = \pinfty,
    \end{align*}
    where
    \begin{align*}
        \Gamma_\lambda \eqdef \sup\bLa H_e(\Pi) \quad|\quad \text{$\Pi$ cycle : $\gamma(\Pi)\ge\gamma^\star+\lambda$}\bRa
    \end{align*}
    and $H_e(\Pi)$ is called the \emph{exit height} of the cycle $\Pi$ and depends on the cycle decomposition.
\end{theorem}

We do not include here the general definition of $H_e$, which can be interpreted as the minimal cost to leave a cycle. In the case of singletons, which are the starting cycles of the cycle construction, $H_e(\{x\})=r^\star(x)$.
Applying \cref{thm:trouvé_convergence} with small enough $\lambda$ yields the following result.
\begin{corollary}
    \label{thm:trouvé_convergence_0}
    Let $(\epsilon_k)_{k\ge0}$ be a non-increasing cooling schedule.
    Let $(X_k)_{k\ge0}$ be an inhomogeneous Markov process following $(P^{\epsilon_k})_{k\ge0}$ with any starting distribution. Then,
    \begin{align*}
        \lim_{k\to\pinfty}\PP(X_k\notin\cX^\star) = 0
        \quad \text{if and only if} \quad
        \sum_{k\ge0} \epsilon_k^{\Gamma_0} = \pinfty,
    \end{align*}
    where
    \begin{align*}
        \Gamma_0 \eqdef \sup\bLa H_e(\Pi) \quad|\quad \text{$\Pi$ cycle : $\gamma(\Pi)>\gamma^\star$}\bRa.
    \end{align*}
\end{corollary}

\Cref{thm:trouvé_convergence_0} states that the process with dynamic perturbation converges to $\cX^\star$ if and only if the sequence $(\epsilon_k^{\Gamma_0})_k$ has infinite sum. In particular, the sequence $(\epsilon_k)_k$ must not decrease faster than a polynomial to ensure convergence. 
It is also shown that there exists a cooling schedule with optimal convergence speed.
\begin{theorem}[\cite{trouvé1996}, Theorem 6.3]
    \label{thm:trouvé_convergence_rate}
     Let $\lambda>0$. There exists a constant $c>0$ such that for all finite horizon $N$, there exists a non-increasing cooling schedule $(\epsilon_k(N))_{0\le k\le N}$ such that if $(X_k)_{k\ge0}$ is an inhomogeneous Markov process following $(P^{\epsilon_k(N)})_{0\le k\le N}$ with any starting distribution,
     \begin{align*}
         \PP(X_N\in\cX_\lambda) \le \frac{c}{N^{\hat{\alpha}_\lambda}},
     \end{align*}
     where
     \begin{align*}
        \hat{\alpha}_\lambda \eqdef \min\La\frac{\max(\gamma(\Pi)-\gamma^\star, \lambda)}{H_e(\Pi)} \quad|\quad \text{$\Pi$ cycle : $\gamma(\Pi)>\gamma^\star$}\Ra.
    \end{align*}
\end{theorem}

The above convergence speed becomes optimal when $\lambda$ is close to $0$.
Applying \cref{thm:trouvé_convergence_rate} with small enough $\lambda$ yields the following result.
\begin{corollary}
    \label{thm:trouvé_convergence_rate_0}
     There exists a constant $c>0$ such that for all finite horizon $N$, there exists a non-increasing cooling schedule $(\epsilon_k(N))_{0\le k\le N}$ such that if $(X_k)_{k\ge0}$ is an inhomogeneous Markov process following $(P^{\epsilon_k(N)})_{0\le k\le N}$ with any starting distribution,
     \begin{align*}
         \PP(X_N\notin\cX^\star) \le \frac{c}{N^{\hat{\alpha}_0}}
     \end{align*}
     where
     \begin{align*}
        \hat{\alpha}_0 \eqdef \min\La\frac{\gamma(\Pi)-\gamma^\star,}{H_e(\Pi)} \quad|\quad \text{$\Pi$ cycle : $\gamma(\Pi)>\gamma^\star$}\Ra.
    \end{align*}
\end{corollary}

\subsection{Application to ITEL}
\label{sec:eps_k:application}

The above results apply to all \acp{rpmp}, hence in particular to \ac{itel}.
They mostly serve a theoretical purpose: \cref{thm:trouvé_convergence_0} hints that the cooling schedule should be of order $\epsilon_k = k^{-1/\Gamma_0}$ at the fastest.
Moreover, the optimal convergence speed given by \cref{thm:trouvé_convergence_rate_0} is of order $N^{-\hat{\alpha}_0}$ for finite horizon $N$.
The precise definition of $\Gamma_0$, $\hat\alpha_0$ and of the optimal cooling schedule are complex and can be found in the proof of \cite[Theorem 7.1]{catoni1992} on which \cite[Theorem 6.3]{trouvé1996} is based.
These definitions involve problem-dependent quantities and their precise computation in the case of \ac{itel} is beyond the scope of this paper.
The cooling rate and the convergence speed are both of polynomial order, which is not very fast in practice. In particular, an exponentially decreasing cooling schedule may give better result, although it will theoretically not converge.

\section{Conclusion}

In this paper, we introduced a more efficient algorithm for distributed learning in the case where rewards are bounded. We also extended the scope of the algorithm to random games with minor assumptions of the utility distributions and showed that with appropriate parameters the algorithm keeps its convergence properties.
Finally, we briefly discussed the possibility of a dynamic perturbation to obtain a stronger sense of convergence towards optimal states.

\printbibliography

\appendix
\newpage
\section{Proof of ITEL convergence}
\label{sec:itel:proof}

In this section we prove the results stated in \cref{sec:itel:results}, that is \cref{prop:itel:recurrence_P_0,prop:itel:recurrence_P_eps,thm:itel:X_star}.
Let $P^\varepsilon$ be the Markov process defined by the \ac{itel} algorithm for any $\varepsilon\in[0,1)$.
As discussed in \cref{sec:intro}, the computation of the potential $\gamma$ is done over recurrence classes of the unperturbed process $P^0$ instead of the whole Markov chain.
Therefore, the first step of the proof is to identify these recurrence classes. The second step is to compute resistances and then potentials in order to describe $\cX^\star$.
The proof given in this section follows the same reasoning as in \cite{TEL}, although the organization of lemmas and some other details differ slightly.

\subsection{Recurrence Classes of the Unperturbed Process}
\label{sec:itel:proof_P_0}

The goal of this section is to study paths of zero resistance, \ie, paths that exist in the unperturbed process.
In particular, we are going to show that in $P^0$ any state can reach a state in $\cC\cup\{D\}$, and that each state $x\in\cC$ is absorbing.
This then implies \cref{prop:itel:recurrence_P_0}: the recurrence classes of $P^0$ are the singletons $\{x\}\subset\cC$ and possibly the communication class of $D$.
First of all, let us show that from any state there is a path to a state without intermediate moods \hopeful{} and \watchful{} or non-aligned benchmarks.
\begin{lemma}
    \label{lemma:itel:path_to_aligned}
    In the unperturbed process $P^0$ there is a path from any state to a state where all agents are either content with aligned benchmark or discontent, and so that discontent agents remain so along the path.
\end{lemma}

Recall that discontent agents store no benchmark actions.
When saying that a state $x=(\blm,\bla,\blu)$ featuring discontent agents is aligned in \cref{lemma:itel:path_to_aligned}, we mean that $\bla$ can be extended to discontent agents in a way that $\blu$ is aligned with $\bla$.

\begin{proof}
Let $x=(\blm,\bla,\blu)\in\cX$ be any state and extend $\bla$ with arbitrary actions for discontent agents.
Denote $\bu$ the utilities resulting from $\bla$, which may differ from $\blu$ if the latter is not aligned with $\bla$.
Let every agent play according to $\bla$ for two steps.
This is possible according to \cref{tab:itel:action_policy}. Indeed, in the unperturbed process, a non-discontent agent always play its benchmark action. Moreover, a discontent agent plays at random uniformly among all actions, hence has a positive probability of playing according to $\bla$ twice in a row.
agents observe $\bu$ for two steps and the following behaviors may happen with positive probability according to \cref{tab:itel:update_policy}:
\begin{itemize}
    \item Discontent agents remain discontent.
    \item Watchful agents with $u<\lu$ reject after one step and end up discontent.
    \item Content or hopeful agents with $u<\lu$ become watchful, then discontent.
    \item Content agents with $u=\lu$ remain in their content state.
    \item Hopeful agents with $u\ge \lu$ and watchful agents with $u=\lu$ accept the outcome after one step and end up content.
    \item Content or watchful agents with $u>\lu$ become hopeful, then content. 
\end{itemize}
Therefore, there exists in $P^0$ a path with positive probability from $x$ to a state $y=(\blm',\bla,\bu)$ where all agents are content or discontent, and $\bu$ is aligned with $\ba$.
\end{proof}

Using the interdependence assumption \cref{assum:interdependence}, it is possible to show that discontent agents can always explore in a way that makes other agents become discontent.
Eventually, if at least one agent is discontent then it can bring all other agents to a discontent mood.
\begin{lemma}
    \label{lemma:itel:path_to_D}
    In the unperturbed process $P^0$ there is a path to $D$ from any state featuring at least one discontent agent.
\end{lemma}

\begin{proof}
Let $x=(\blm,\bla,\blu)\in\cX$ be a state with at least one discontent agent.
By first applying \cref{lemma:itel:path_to_aligned}, we can assume that $x$ only features content or discontent agents and that $\bla$ is extended to discontent agents and yields the utilities $\blu$.
Let $J$ be the set of discontent agents and $a_J\ne\la_J$ be any other actions for the discontent agents. Denote $\bu$ the utilities resulting from $(a_J,\la_{-J})$.
Assume that every agent play according to $(a_J,\la_{-J})$ for two steps, then goes back to playing $\bla$ for another two steps.
Agents then observe $\bu$ for two steps then $\blu$ for the other two steps, and the following behaviors may happen with positive probability:
\begin{itemize}
    \item Discontent agents remain discontent.
    \item Content agents with $u<\lu$ become watchful then discontent. They remain discontent for the following two steps.
    \item Content agents with $u>\lu$ become hopeful then accept $u$ as their new benchmark. They then become watchful then discontent, as their new benchmark is $u$ but they observe $\lu$.
    \item Content agents with $u=\lu$ remain in their content state.
\end{itemize}
As in \cref{lemma:itel:path_to_aligned}, this has positive probability to happen in $P^0$ according to \cref{tab:itel:policies}.
Now, using \cref{assum:interdependence} over the set of discontent agents in $x$, one can choose $J$ and $a_J$ such that $u_i\ne\lu_i$ for some content agent $i$. In this case, at least one content agent fits one of the last two categories, and therefore eventually becomes discontent along the described path.
We have shown that there is a path in $P^0$ from $x$ to another state with strictly more discontent agents. Iterating the reasoning until all agents are discontent eventually shows that there is a path in $P^0$ from the original state $x$ to $D$.
\end{proof}

Combining both \cref{lemma:itel:path_to_aligned,lemma:itel:path_to_D} concludes that in $P^0$, there is a path from any state to either $D$ or a state in $\cC$.
Indeed, there is a first path to a state where all agents are either discontent or content and aligned. If this state is not in $\cC$ then there is at least one discontent agent hence a path to $D$.
This implies that the recurrence classes of $P^0$ are included in the communication classes of $\cC$ and $D$.
Since there are no explorations from content agents in $P^0$, any state $\{x\}\subset\cC$ is absorbing, hence a recurrence class of $P^0$.
This proves \cref{prop:itel:recurrence_P_0}.

Regarding whether or not $D$ is recurrent, allowing $F$ to be equal to $0$ makes the discussion more complex than in \ac{tel}. In \ac{tel}, a discontent agent always stays so, hence $D$ is be absorbing and $\{D\}$ is a recurrence class of $P^0$.
In \ac{itel}, it is possible that the communication class of $D$ contains other states as a discontent agent accepts an exploration when the observed utility $u$ satisfies $F(u)=0$.
Furthermore, it is possible that $D$ is not recurrent if all agents can simultaneously observe a utility satisfying $F(u)=0$, as this implies the existence of a path from $D$ to an aligned state, which is absorbing.
Whether or not the class of $D$ is recurrent does not influence the following.

From now on we consider the resistance graph $\cG$ over the class of $D$ and classes $\{x\}\subset\cC$. We remove the brackets and refer to them as $D$ and $x\in\cC$ to ease notations. Note that if $D$ is not recurrent then its outward resistance is $r^\star(D)=0$ and it follows immediately that it cannot minimize $\gamma$.

\subsection{Resistances and Potentials}
\label{sec:itel:proof_res}

Now that we have identified the recurrence classes of $P^0$, we compute the resistance between these classes in the perturbed process $P^\varepsilon$ and then their potential.
From now on we reason over the graph of the classes $x\in\cC$ and $D$, where the resistance of any edge $x\to y$ is the lowest resistance of a path $x\leadsto y$ in $P^\varepsilon$.
Recall that $\cE\subset\cC$ is the set of equilibrium states and $\cA\subset\cE$ is the subset of optimal states.
Definitions of easy edges, rooted trees, and potentials were given in \cref{def:resistance,def:potential}.
Definitions of virtual welfare and stability were given in \cref{def:itel:tW&tS}.

It is not necessary to compute all possible resistances in $\cG$ to compute potentials, as we are only interested in minimal rooted trees. In fact, it is sufficient to study the resistances of edges $D\to x$ along with the easy edges of states $x\in\cC$.
\begin{lemma}
    \label{lemma:itel:r_x}
    \hfill
    \begin{enumerate}[(i)]
       \item If $x\in\cA$, $r^\star(x) = \pinfty$.
       \item If $x\in\cE\setminus\cA$, $r^\star(x) = 2 = r(x\to D)$.
       \item If $x\in\cC\setminus\cE$, $r^\star(x) = 2 - \tS(x)$, and if $x\to D$ is not easy then any easy edge $x\to y\in\cC$ leads to a strictly better welfare, that is $W(y)>W(x)$.
    \end{enumerate}
\end{lemma}

\begin{proof}
Let $x=(\blm,\bla,\blu)\in\cC$.

\medskip\noindent\textbf{(i).}
When $x$ is optimal, no agent can explore, so that $\bla$ is played continuously and so is observed $\blu$. Agents never change state, and the state is absorbing, \ie, $r^\star(x) = \pinfty$.

\medskip\noindent\textbf{(ii).}
When $x$ is an equilibrium, all agents are at an equilibrium, hence cannot accept their exploration if they are the only agent to explore as the observed utility would be lower than their benchmark.
Hence, a single exploration cannot leave the recurrence class of $x$: other agents may become hopeful or watchful but then would require another exploration to change state, else they revert to their original state.
Therefore at least two explorations, either successive or simultaneous, are needed to leave the recurrence class of $x$, so that $r^\star(x)\ge2$.

Conversely, one can construct a path $x\leadsto D$ of resistance equal to $2$ using a reasoning almost identical to that of \cref{lemma:itel:path_to_D}.
Indeed, consider the same path as the one described, except that one agent $j$ is chosen as a content agent at an equilibrium but not optimized---which exists since $x\in\cE\setminus\cA$---instead of $J$ the set of discontent agents.
Let $j$ play some action $a_j\ne\la_j$ twice in a row then go back to playing $\la_j$ while all other agents play according to $\la$. As $j$ is at an equilibrium, its observation $u_j$ in the action profile $(a_j,\la_{-j})$ is lower than its benchmark $\lu_j$. Therefore $j$ always reverts after playing $a_j$.
This choice of actions implies a resistance of $2$, as two explorations are performed.
As in \cref{lemma:itel:path_to_D}, agents observe $\bu$ for two steps then $\blu$ for the other two steps, and the following behaviors may happen with positive probability:
\begin{itemize}
    \item $j$ reverts twice then remains in its content state.
    \item Content agents with $u<\lu$ become watchful then discontent. They remain discontent for the following two steps.
    \item Content agents with $u>\lu$ become hopeful then accept $u$ as their new benchmark. They then become watchful then discontent, as their new benchmark is $u$ but they observe $\lu$.
    \item Content agents with $u=\lu$ remain in their content state.
\end{itemize}
According to \cref{tab:itel:update_policy}, the behaviors described above happen with no additional resistance under the aforementioned choices of actions.
Now, using the interdependence assumption \cref{assum:interdependence}, one can choose $a_j$ such that $u_i\ne\lu_i$ for an agent $i\ne j$. In other words, at least one agent fits in one of the last two categories, hence eventually becomes discontent along the described path.
We have shown that there is a path of resistance equal to $2$ from $x$ to another state with at least one discontent agent. We conclude with \cref{lemma:itel:path_to_D} that with no additional resistance the path can be extended to $D$.
Therefore, $r(x\to D) = 2 = r^\star(x)$.

\medskip\noindent\textbf{(iii).}
When $x$ is not an equilibrium, one exploration is still not enough to change state, but an agent $i$ that is not an at an equilibrium may accept the outcome of its exploration with some resistance controlled by $G$.
Let $i$ be such agent, and let $a_i$ be an action such that the observed utility $u_i=U_i(a_i,\la_{-i})$ when $i$ explores $a_i$ satisfies $u_i>\lu_i$.
The resistance of accepting this exploration is $G(\lu_i,u_i)$, which is exactly $1-\tS(x)$ according to \cref{def:itel:tW&tS} when $i$ and $a_i$ are chosen as to minimize $G(\lu_i,u_i)$. Adding the cost of exploration we get the lower bound $r^\star(x) \ge 2-\tS(x)$.
Note that this bound is strictly lower than $2$ as $\tS(x)>1-G_0\ge0$ under \cref{cond:itel}.

Let us now show that accepting this exploration is sufficient to change state, so that $r^\star(x) = 2-\tS(x)$.
Consider the path similar to the first two steps described in (ii) except that the explorer is $i$ and it accepts its exploration after the first step and plays it again instead of reverting and exploring again.
This causes the others content agents to either become discontent or to remain content with a utility $u_j\ge\lu_j$.
If all agents remain content, the new state $y$ is an aligned state and all agents have either improved in utility or observed no change, hence $W(y)>W(x)$ since at least $i$ did improve.
Else there is a path of zero resistance to $D$ according to \cref{lemma:itel:path_to_D}.
Overall this path has an associated resistance of $2-\tS(x)$ due to $i$ exploring and accepting. We conclude that $r^\star(x) = 2-\tS(x)$.
\end{proof}

\begin{lemma}
    \label{lemma:itel:r_D}
    For all $x\in\cC$, $r(D\to x) = 1 - \tW(x)$.
\end{lemma}

\begin{proof}
Let $x=(\blm,\bla,\blu)\in\cC$ and consider any path $D\leadsto x$.
Each agent $i$ eventually ends up content with utility $\lu_i$. Since a content agent cannot decrease in benchmark utility without first getting discontent, there is a point in the path where $i$ goes from discontent to content by accepting a utility $u_i\le\lu_i$. This implies a resistance of $F(u_i)\ge F(\lu_i)$ as $F$ is non-increasing.
Each agent does such acceptation at some point in the path, hence the total resistance of the path is at least $\sum_i F(\lu_i) = 1-\tW(x)$.

Conversely, there is a direct path $D\to x$ with resistance $\sum_i F(\lu_i)$ when all agents choose to play according to $\bla$ simultaneously and to accept the outcome, hence $r(D\to x) = 1-\tW(x)$.
\end{proof}

Before moving on to computing potentials, notice that the previous lemmas allow us to conclude on the nature of the \ac{itel} process.
First of all, states $x\in\cA$ are absorbing even in the perturbed process according to \cref{lemma:itel:r_x}.
Moreover, there is a path in $P^\varepsilon$ with finite resistance from $D$ to any state $x\in\cC$ due to \cref{lemma:itel:r_D}.
We can show that there is also a path from $x$ back to either $D$ or a state in $\cA$ when $x\notin\cA$. Indeed, \cref{lemma:itel:r_x} shows that an easy edge leaving $x$ leads either to $D$ or to another state in $\cC$ with strictly higher welfare. Since there is a finite amount of states in $\cC$, following easy edges eventually leads to $D$ or to $\cA$. Indeed, if it was not the case, the easy edges would create a cycle $x=x_1\to x_2\to \dots\to x_m=x$ at one point, but this cycle would satisfy $W(x)=W(x_1)<W(x_2)<\dots<W(x_m)=W(x)$, which is a contradiction.

Therefore, if $\cA\ne\emptyset$, there is a path from any state to a state in $\cA$ which is absorbing. Hence the perturbed process is not irreducible and its recurrence classes are exactly the singletons $\{x\}\subset\cA$.
On the other hand, when $\cA=\emptyset$, $D$ communicates with all states in $\cC$, so that all these states belong to the same recurrence class. Hence the perturbed process consists of a unique recurrence class, \ie, is irreducible\footnote{There may actually be states that are not be part of the recurrence class of $P^\varepsilon$. In this case we restrict ourselves to the study of the single recurrence class, which is a \ac{rpmp}. This is without loss of generality, as the algorithm ends up \ac{as} in this recurrence class regardless of the starting state.}. It is also aperiodic as the transition $D\to D$ has positive probability.
This concludes \cref{prop:itel:recurrence_P_eps}.
When $\cA\ne\emptyset$, the process converges \ac{as} to one of its recurrence classes, which is one of the absorbing states $x\in\cA$.

The rest of this section is devoted to the case where $\cA=\emptyset$, where \cref{thm:sss_min_gamma} can be applied.
We now study minimal rooted trees for each state, in order to compute their potential.
\begin{lemma}
    \label{lemma:itel:potential}
    \hfill
    \begin{enumerate}[(i)]
        \item $\gamma(D) = \sum_{x\in\cC} r^\star(x)$.
        \item $\gamma(x) = \gamma(D) - r^\star(x) + r(D\to x) =
        \begin{cases}
            \gamma(D) - 1 - \tW(x) & \text{if $x\in\cE$,}\\
            \gamma(D) - 1 - \tW(x) +  \tS(x) & \text{if $x\in\cC\setminus\cE$.}
        \end{cases}$
    \end{enumerate}
\end{lemma}

\begin{proof}
Let $x=(\blm,\bla,\blu)\in\cC$.

\medskip\noindent\textbf{(i).}
We want to show that it is possible to create a $D$-tree using only easy edges.
For each $x\in\cC$, choose an easy edge from $x$, going to $D$ if possible, and consider the resulting sub-graph with only the chosen edges.
According to \cref{lemma:itel:r_x}, the chosen easy edge either leads directly to $D$ or leads to a state of strictly better welfare.
We already argued that this choice of edges cannot create a cycle.
Therefore, the sub-graph is acyclic with exactly one outward edge for each $x\in\cC$, hence a $D$-tree. We conclude that $\gamma(D) = \sum_{x\in\cC} r^\star(x)$ as the potential cannot be smaller.

\bigskip\noindent\textbf{(ii).}
It is immediate that in the previous tree, removing the edge leaving $x$ and adding $D\to x$ instead gives a $x$-tree of resistance $\gamma(D) - r^\star(x) + r(D\to x)$.
Let us now show that this $x$-tree is minimal. Consider any $x$-tree $T$ and the path $D\to y_1\to y_2\to \dots\to y_m\to x$ in $T$. Denote $\bu$ the utilities of $x$.
Reusing the proofs of \cref{lemma:itel:r_x,lemma:itel:r_D}, we know that along this path:
\begin{itemize}
    \item As shown in \cref{lemma:itel:r_x}, any path leaving $y_k$ has a resistance lower bounded by $r^\star(y_k)$ due to explorations and acceptations from content agents.
    \item As shown in \cref{lemma:itel:r_D}, each agent $i$ must accept at some point a utility $u_i\le\lu_i$ from a discontent state, for a resistance $F(u_i)\ge F(\lu_i)$.
\end{itemize}
Both points involve resistances of different natures. The resistances due to content agents exploring or accepting do not intersect with the resistances due to discontent agents accepting.
Hence we deduce from these observations that the total resistance of the path is greater than the sum of all the resistance mentioned above, that is $\sum_{k=1}^m r^\star(y_k) + \sum_i F(\lu_i)$.
Lower bounding the resistances of the other edges of $T$ via $r^\star$, we have \[r(T) \ge \sum_{y\in\cC\setminus\{x\}} r^\star(y) + \sum_i F(\lu_i) = \gamma(D) - r^\star(x) + r(D\to x)~.\]
This proves that $\gamma(x) = \gamma(D) - r^\star(x) + r(D\to x)$.
The specific formulas depending on whether $x\in\cE$ or not are derived directly from \cref{lemma:itel:r_x}.
\end{proof}

We can now identify the states minimizing $\gamma$ using the bounds on $F$ and $G$ from \cref{cond:itel}, we recall:
\begin{align*}
    \begin{cases}
        0 \le F \le \frac{F_0}{n},\\
        0 \le G < G_0,\\
        F_0 + G_0 \le 1.
    \end{cases}
\end{align*}
\Cref{cond:itel} directly implies that $1-F_0 \le \tW \le 1$ and $1-G_0 < \tS \le 1$, hence
\begin{align*}
    0 \le 1-\tW < \tS \le 1
\end{align*}
holds for any possible value of $\tS$ and $\tW$.
It follows that if $x\in\cE$ and $y\in\cC\setminus\cE$, according to \cref{lemma:itel:potential},
\begin{align*}
    & \gamma(x)
    = \gamma(D) - 1 - \tW(x)
    < \gamma(D) - 2 + \tS(y)
    \le \gamma(D) - 1 - \tW(y) + \tS(y)
    = \gamma(y)~,\\
    & \gamma(y)
    = \gamma(D) - 1 - \tW(y) + \tS(y)
    < \gamma(D) - 1 + \tS(y)
    \le \gamma(D)~.
\end{align*}
Moreover, $\gamma$ is minimized over $\cE$ at states maximizing $\tW$, and minimized over $\cC\setminus\cE$ at states maximizing $\tW-\tS$.
We deduce that, if $\cE\ne\emptyset$, $\cX^\star$ is the set of equilibrium states maximizing $\tW$.
Else, if $\cE=\emptyset$, $\cX^\star$ is the set of aligned states maximizing $\tW-\tS$.
This concludes the proof of \cref{thm:itel:X_star}.

\section{Proof of IODL convergence}
\label{sec:iodl:proof}

The proof of convergence for \ac{iodl} uses the same reasoning as for \ac{itel}. In this section we highlight the few differences due to the removal of intermediate moods \hopeful{} and \watchful{} which yield \cref{thm:iodl:convergence}.

All of the discussion done in \cref{sec:itel:proof_P_0} also holds for \ac{iodl}. The paths described must be adapted by removing intermediate step, but the reasoning is identical.
In particular, \cref{lemma:itel:path_to_aligned,lemma:itel:path_to_D,prop:itel:recurrence_P_0} hold for the \ac{iodl} process.
Regarding the computation of resistances, \cref{lemma:itel:r_x} is adapted as follows.
\begin{lemma}
    \label{lemma:iodl:r_x}
    \hfill
    \begin{enumerate}[(i)]
       \item If $x\in\cA$, $r^\star(x) = \pinfty$.
       \item If $x\in\cC\setminus\cA$, $r^\star(x) = 1$, and if $x\to D$ is not easy than any easy edge $x\to y\in\cC$ leads to a strictly better welfare: $W(y)>W(x)$.
    \end{enumerate}
\end{lemma}

\begin{proof}
When $x\in\cA$, the reasoning is the same as in \cref{lemma:itel:r_x}.
When $x\in\cC\setminus\cA$, leaving $x$ requires some exploration hence $r^\star(x)\ge1$.
Now let us describe a path leaving $x$ with resistance equal to $1$.
By interdependence assumption \cref{assum:interdependence}, one can find a non-optimized agent $j$ that is able to explore in a way that affects other agents.
First, let us assume that $j$ is able to do so in a way so that $G(\lu_j,u_j)>0$, so that it can reverts with no resistance.
Let $j$ explore once than playing its benchmark action. According to \cref{tab:iodl:policies} and with the same arguments that were used in \cref{lemma:itel:r_x} (ii), the following path may happen with resistance equal to $1$ due to a single observation from $j$:
\begin{itemize}
    \item $j$ reverts once then remains in its content state.
    \item Content agents with $u<\lu$ become discontent. They remain discontent after the second step.
    \item Content agents with $u>\lu$ accept $u$ as their new benchmark. They then become discontent, as their new benchmark is $u$ but they observe $\lu$.
    \item Content agents with $u=\lu$ remain in their content state.
\end{itemize}
Along this path some agents become discontent. We conclude with \cref{lemma:itel:path_to_D} that with no additional resistance the path can be extended to $D$.
Therefore, $r(x\to D) = 1 = r^\star(x)$.
Now, in the case where $G(\lu_j,u_j)=0$, $j$ accepts its exploration with no resistance. Then we conclude with the same reasoning as in \cref{lemma:itel:r_x} (iii) that there the path is extended with no additional resistance to $D$ or to another state $y\in\cC$ with $W(y)>W(x)$.
\end{proof}

\Cref{lemma:itel:r_D} also holds for \ac{iodl} with same reasoning, we recall:
\begin{lemma}
    \label{lemma:iodl:r_D}
    For all $x\in\cC$, $r(D\to x) = 1-\tW(x)$.
\end{lemma}

One can see that \cref{prop:itel:recurrence_P_eps} holds in the case of \ac{iodl} for the same reasons as for \ac{itel}. The case where $\cA\ne\emptyset$ in \cref{thm:iodl:convergence} follows as for \ac{itel}.
When $\cA=\emptyset$, the same reasoning as in \cref{lemma:itel:potential} yields the following.
\begin{lemma}
    \label{lemma:iodl:potential}
    \hfill
    \begin{enumerate}[(i)]
        \item $\gamma(D) = \sum_{x\in\cC} r^\star(x)$.
        \item If $x\in\cC$, $\gamma(x) = \gamma(D) - r^\star(x) + r(D\to x) = \gamma(D) - \tW(x)$.
    \end{enumerate}
\end{lemma}

Now recall that we assumed that $F<1/n$, hence $\tW(x) = 1 - \sum_iF(\lu_i) \in (0,1]$ for any state $x\in\cC$ with benchmark utilities $\blu$.
From there, it is immediate with \cref{lemma:iodl:potential} that $\gamma$ is minimized at states $x\in\cC$ maximizing $\tW(x)$, and applying \cref{thm:sss_min_gamma} concludes \cref{thm:iodl:convergence}.

\section{Proof of RITEL convergence}
\label{sec:ritel:proof}

In this section we prove all results stated in \cref{sec:ritel:results}, that is \cref{prop:ritel:recurrence_P_0,prop:ritel:recurrence_P_eps,thm:ritel:X_star,thm:ritel:convergence}.
Let $P^\varepsilon$ be the Markov process defined by the \ac{ritel} algorithm for any $\varepsilon\in[0,1)$.
As for \ac{itel}, the first step of the proof is to identify the recurrence classes of the unperturbed process $P^0$. The second step is to compute resistances and then potentials in order to describe $\cX^\star$.
Contrarily to \ac{itel}, resistances are estimated with a small margin of error, so that we cannot identify exactly which states minimize $\gamma$, as stated in \cref{thm:ritel:convergence}.
The proof detailed in this section is very similar to \cref{sec:itel:proof}. In fact, most adaptations have already been discussed in \cref{sec:ritel:results} and it only remains to adapt the proofs from \ac{itel}.

\subsection{Recurrence Classes of the Unperturbed Process}
\label{sec:ritel:proof_P_0}

The goal of this first section is to study paths of zero resistance, \ie, paths in the unperturbed process.
In particular, we show that in $P^0$ any state can reach a state in $\cC_\delta\cup\{D\}$, and that each state $x\in\cC_\delta$ is absorbing.
From there we deduce \cref{prop:ritel:recurrence_P_0}, that is the recurrence classes of $P^0$ are the singletons $\{x\}\subset\cC_\delta$ and possibly the communication class of $D$.
Apart from $\cC$ being replaced with $\cC_\delta$, this section is proven exactly as in \cref{sec:itel:proof_P_0} and \cref{lemma:ritel:path_to_aligned,lemma:ritel:path_to_D} hold the same statements as \cref{lemma:itel:path_to_aligned,lemma:itel:path_to_D}. We prove them by referring to the \ac{itel} case, highlighting the few required adaptations. Recall that the policies of \ac{ritel} are given in \cref{tab:ritel:policies}.

\begin{lemma}
    \label{lemma:ritel:path_to_aligned}
    In the unperturbed process $P^0$ there is a path from any state to a state where all agents are either content with weakly aligned benchmark or discontent, and so that discontent agents remain so along the path.
\end{lemma}

\begin{proof}
Let $x=(\blm,\bla,\blv)\in\cX$ be any state and extend $\bla$ with arbitrary actions for discontent agents.
Denote $\bv$ the mean bins resulting from $\bla$.
We have shown in \cref{corol:almost_regular_deviation} that the resistance of observing the mean bin is zero. Hence the event where all agents play according to $\bla$ and observe $\bv$ has positive probability to happen in $P^0$.
Then, the path described \cref{lemma:itel:path_to_aligned}---replacing comparisons $u=\lu$ (resp. $u<\lu$ and $u>\lu$) by $v=\lv\pm\delta$ (resp. $v\le\lv-2\delta$ and $v\ge\lv+2\delta$)--- may happen here with positive probability. This path leads to a state where agents are either discontent or content. Moreover, content agents either keep their benchmark if it is already weakly aligned with $\bla$, or accept $v$. Either way content agents all end up weakly-aligned with $\bla$ in this new state, which concludes the proof.

Regarding the special case where the mean utility lies exactly at the edge between bins $v-\delta$ and $v$ and the utility distribution is not deterministic, recall that both bins are observed with probability $\frac12$ in $P^0$ and a benchmark $\lv=v+\delta$ is not considered weakly aligned in this case according to \cref{def:weak&strong_aligned}. We assume that agents in such situation observe $v-\delta$ instead of $v$, so that they act as if they observed a lower utility and eventually become discontent.
\end{proof}

\begin{lemma}
    \label{lemma:ritel:path_to_D}
    In the unperturbed process $P^0$ there is a path to $D$ from any state featuring at least one discontent agent.
\end{lemma}

\begin{proof}
Let $x=(\blm,\bla,\blv)\in\cX$ be a state with at least one discontent agent.
By first applying \cref{lemma:ritel:path_to_aligned}, we can assume that $x$ only features content or discontent agents.
Let $J$ be the set of discontent agents and $a_J\ne\la_J$ be any other actions for the discontent agents. Denote $\bnu$ the mean bins of the utilities resulting from $(a_J,\la_{-J})$.
Then the path described in \cref{lemma:itel:path_to_D}---replacing comparisons $u=\lu$ (resp. $u<\lu$ and $u>\lu$) by $\nu=\lv\pm\delta$ (resp. $\nu\le\lv-2\delta$ and $\nu\ge\lv+2\delta$)---may happen here with positive probability and leads to another state with more discontent agents, assuming that $a_J$ is chosen so that some non discontent agent $i$ observes $\nu_i\ne\lv_i\pm\delta$. 

It remains to show that the stronger interdependence assumption \cref{assum:3dinterdependence} enables such action profile.
Indeed, according to \cref{assum:3dinterdependence}, there exist $a_J$ and $i\notin J$ such that $M_i(a_J,\la_{-J}) \le M_i(\bla)-3\delta$ or $M_i(a_J,\la_{-J}) \ge M_i(\bla)+3\delta$. In the first case, $\nu_i^- \le M_i(a_J,\la_{-J}) \le M_i(\bla)-3\delta < \lv_i^--\delta$ as weak alignment implies $M_i(\bla)\in[\lv_i^--\delta,\lv_i^-+2\delta)$, hence $\nu_i < \lv_i-\delta$, \ie, $\nu_i \le \lv_i-2\delta$. Similarly, $\nu_i \ge \lv_i+2\delta$ in the second case.
We conclude that the path described above may happen with positive probability in $P^0$. Iterating until all agents are discontent concludes the proof.
\end{proof}

As for \ac{itel}, \cref{lemma:ritel:path_to_aligned,lemma:ritel:path_to_D} imply that in $P^0$, there is a path from any state to either $D$ or a state in $\cC_\delta$.
In particular, the recurrence classes of $P^0$ are included in the communication classes of $\cC_\delta$ and $D$.
In a weakly aligned state, all agents are content hence keep playing according to their benchmark in $P^0$. Moreover, being weakly aligned, the observed utilities always fall within $\delta$ of their benchmark bin in $P^0$. Therefore, all agents remain in their state and the state is absorbing, so that each $\{x\}\subset\cC_\delta$ is a recurrence class of $P^0$.
This proves \cref{prop:ritel:recurrence_P_0}.

Regarding whether or not $D$ is recurrent, the same discussion as in \ac{itel} holds: the communication class of $D$ is not necessarily reduced to $D$ itself, and may not be recurrent.
Either way from now on we consider the resistance graph $\cG$ over the class of $D$ and classes $\{x\}\subset\cC_\delta$. We remove the brackets and refer to them as $D$ and $x\in\cC_\delta$ to ease notations. Note that if $D$ is not recurrent then its outward resistance is $r^\star(D)=0$ and it follows immediately that it cannot minimize $\gamma$.

\subsection{Resistances and Potentials}
\label{sec:ritel:proof_res}

Now that we have identified the recurrence classes of $P^0$, we compute the resistance between these classes in the perturbed process $P^\varepsilon$ and then their potential.
From now on we reason over the graph of the classes $x\in\cC_\delta$ and $D$, where the resistance of any edge $x\to y$ is the lowest resistance of a path $x\leadsto y$ in $P^\varepsilon$.
Recall that subsets $\cC_\delta$, $\cC$, $\cE_\delta$, $\cE$, $\cA_\delta$ and $\cA$ were defined in \cref{sec:ritel:results}.
Definitions of easy edges, rooted trees, and potentials were given in \cref{def:resistance,def:potential}.
Definitions of virtual welfare and stability of a state were given in \cref{def:ritel:tW&tS_state}.

As for \ac{itel}, it is enough to study the resistances of edges $D\to x$ along with the easy edges of states $x\in\cC_\delta$.
Before computing the outward resistances from states $x\in\cC_\delta$, we discuss the nature of an easy edge leaving $x$.
As in the case of \ac{itel}, it is required in order to leave $x$ that either an agent observes a non-aligned utility twice, or that an agent accepts its own exploration.
However, to observe a non-aligned utility, an agent can be influenced by another agent as in \ac{itel}, or can make an offset observation due to random utilities.
The following lemma discusses the resistance of each three kinds of paths depending on the nature of the starting state $x$. An easy edge leaving $x$ then corresponds to the path with minimal resistance among the ones described.
\begin{lemma}
    \label{lemma:ritel:easy_edge}
    \hfill
    \begin{enumerate}[(i)]
        \item If $x\in\cC_\delta$ and $\tR(x) < \pinfty$, there exists a path from $x$ to $D$ where an agent makes an offset observation twice, with resistance equal to  $2\tR(x)$. If $\tR(x) = \pinfty$ instead the path does not exist.
        \item If $x\in\cE_\delta\setminus\cA_\delta$, there exists a path from $x$ to $D$ where an agent explores twice and influences another agent into changing state, with resistance equal to $2$.
        If $x\in\cA_\delta$, no such path exists.
        \item If $x\in\cC_\delta$, there might exist a path from $x$ to another state where an agent explores and accepts, which leads either to $D$ or to a state $y\in\cC_\delta$ satisfying $W(y)>W(x)$.
        If $x\in\cC_\delta\setminus\cE_\delta$, the minimal resistance among such paths is at most $2-\tS_-(x)$.
        
        If $x\in\cC$, such paths may or may not exist, we lower bound their resistance.
        If $x\in\cC\setminus\cE$, the minimal resistance among such paths is at least $2-\tS_+(x)$.
        If $x\in\cE\setminus\cA_\delta$, the minimal resistance among such paths is at least $2$.
        If $x\in\cA_\delta$, no such path exists.
    \end{enumerate}
\end{lemma}

\begin{proof}
    Let $x=(\blm,\bla,\blv)\in\cC_\delta$.
    
    \medskip\noindent\textbf{(i).}
    If no agent explores, the minimal resistance for an agent to make a non-aligned observation is exactly $\tR(x)$ by \cref{def:ritel:noise_stability}.
    Assume that such observation happens twice in a row while other agents observe their mean bin, which is within $\delta$ of their benchmark bin by weak alignment. The associated resistance is $2\tR(x)$ and the agent making offset observations becomes either watchful then discontent if the observation was lower than its benchmark, or becomes hopeful then accepts the offset observation as its new benchmark if it was higher.
    In the second case, the agent is no longer aligned, hence with no resistance it observes its mean bin which is now perceived as a deterioration. It then becomes watchful, then discontent. Either way the agent eventually becomes discontent and the path can be extended from there to $D$ with no additional resistance due to \cref{lemma:ritel:path_to_D}.
    Note that if $\tR(x)=\pinfty$, the path actually does not exist as \cref{corol:almost_regular_deviation} points out that $r_U(v) = \pinfty$ implies $\PP(U^{(\tau)}\in v)=0$.
    
    \medskip\noindent\textbf{(ii).}
    If $x\in\cE_\delta\setminus\cA_\delta$, the path described in \cref{lemma:itel:r_x} (ii) can be used again and leads to $D$, with the same adaptations as in \cref{lemma:ritel:path_to_aligned,lemma:ritel:path_to_D}. Its resistance is equal to $2$, corresponding to both explorations.
    If $x\in\cA_\delta$, all agents are optimized so no exploration may happen to influence other agents and the path does not exist.
    
    \medskip\noindent\textbf{(iii).}
    Let $i$ be an agent that is not optimized in $x$, $a_i$ an action and $v_i \ge \lv_i+2\delta$ a utility bin such that $r_{U_i(a_i,\la_{-i})}(v_i) < \pinfty$. Then it is possible for $i$ to explore $a_i$ while other agents play according to their benchmark action, to observe $v_i$ and to accept it. This series of events may happen with resistance
    \[1 + r_{U_i(a_i,\la_{-i})}(v_i) + G(\lv_i,v_i)~.\]
    Agent $i$ may play $a_i$ a second time with no resistance as it is its new benchmark action. Meanwhile, any other agent $j$ whose benchmark bin $\lv_j$ is not aligned with its new mean bin $\nu_j = N_j(a_i,\la_{-i})$ may observe a bin $v_j\ne\lv_j\pm\delta$ with no resistance.
    In general this bin would be $v_j = \nu_j$, the only exception being the case of a non deterministic utility with $M_j(a_i,\la_{-i}) = \nu_j^- = \lv_j^--\delta$, in which case $\nu_j=\lv_j\pm\delta$ but the agent may observe $v_j = \nu_j-\delta = \lv_j-2\delta$ with no resistance.
    If all agents observe an improvement, they become hopeful then accept their observation and the new state is a state $y\in\cC_\delta$ with strictly higher welfare.
    Else, some agent becomes discontent and the path can be extended from there to $D$ with no additional resistance due to \cref{lemma:ritel:path_to_D}.
    
    Let us now discuss more precisely when such path may be possible and its resistance. If $x\in\cA_\delta$, the path does not happen as all agents are optimized in $x$ hence do not explore.
    
    If $x\in\cC_\delta\setminus\cE_\delta$, there exists an agent $i$ that is not at a $\delta$-equilibrium in $x$. Then the path is possible with $v_i = N_i(a_i,\la_{-i}) \ge \lv_i+2\delta$, which yields a resistance of $1 + G(\lv_i,v_i)$ as there is no resistance to observing $v_i$.
    Choosing $i$ and $a_i$ such that this resistance is minimal, the total resistance of this path is exactly $2-\tS_-(x)$ according to \cref{def:ritel:tW&tS_state}.
    
    If $x\in\cC\setminus\cE$, there exists an agent $i$ that is not at an equilibrium in $x$. Then the path is possible with any $v_i \ge N_i(a_i,\la_{-i}) \ge \lv_i+\delta$ granted that the resistance of the observation is finite and that $v_i \ge \lv_i+2\delta$. Denote $\nu_i = N_i(a_i,\la_{-i})$.
    For $v_i = \nu_i$, if possible, the path has resistance $1 + G(\lv_i,\nu_i) \ge 1 + G(\lv_i,\nu_i+\delta)$ as there is no resistance to observing $\nu_i$ and $G(\lv_i,\cdot)$ is non-increasing.
    For $v_i = \nu_i+\delta$, the path has resistance $1 + r_{U_i(a_i,\la_{-i})}(\nu_i+\delta) + G(\lv_i,\nu_i+\delta) \ge 1 + G(\lv_i,\nu_i+\delta)$.
    For $v_i \ge \nu_i+2\delta$, the path has resistance $1 + r_{U_i(a_i,\la_{-i})}(v_i) + G(\lv_i,v_i) \ge 1 + R_0 + G(\lv_i,v_i) \ge 2 > 1+G(\lv_i,\nu_i+\delta)$ according to \cref{corol:almost_regular_deviation,cond:R_0,cond:itel}. Regardless of the case, the resistance of a path corresponding to $i$ exploring $a_i$ and accepting is lower bounded by $1+G(\lv_i,\nu_i+\delta)$. Taking the minimum of this lower bound over all possible $i$ and $a_i$ such that $\nu_i\ge\lv_i+\delta$ yields the lower bound $2-\tS_+(x)$ according to \cref{def:ritel:tW&tS_state}.
    
    If $x\in\cE\setminus\cA_\delta$, all agents are at an equilibrium, which implies that any choice of $i$ and $a_i$ would satisfy $\nu_i \le \lv_i$, hence accepting would require an offset observation of at least two bins which according to \cref{corol:almost_regular_deviation,cond:R_0} implies a total resistance greater than $1+R_0 \ge 2$.
\end{proof}

Identifying which of the three paths has least resistance in \cref{lemma:ritel:easy_edge} implies the following bounds on the outward resistances.
\begin{lemma}
    \label{lemma:ritel:r_x}
    \hfill
    \begin{enumerate}[(i)]
        \item If $x\in\cA_\delta$, $r^\star(x) = 2\tR(x)$.
        \item If $x\in\cE\setminus\cA_\delta$, $r^\star(x) = 2$.
        \item If $x\in\cE_\delta\setminus\cA_\delta$, $r^\star(x) \le 2$.
        \item If $x\in\cC\setminus\cE$, $r^\star(x) \ge 2-\tS_+(x)$.
        \item If $x\in\cC_\delta\setminus\cE_\delta$, $r^\star(x) \le 2-\tS_-(x)$.
    \end{enumerate}
    Moreover, in all cases an easy edge leads either to $D$ or to a state $y\in\cC_\delta$ satisfying $W(y)>W(x)$.
\end{lemma}

\begin{proof}
We already justified that an easy edge necessarily corresponds to one of the three paths of \cref{lemma:ritel:easy_edge}, which implies that easy edges will lead to $D$ or to $y\in\cC_\delta$ with $W(y)>W(x)$.
In the following we identify which of the three paths---which we refer to as paths (i), (ii), and (iii) as presented in \cref{lemma:ritel:easy_edge}---has minimal resistance depending on the nature of $x$.

\medskip\noindent\textbf{(i).}
If $x\in\cA_\delta$, \cref{lemma:ritel:easy_edge} states that path (i) is the only way to leave $x$, hence it is the easy edge and $r^\star(x)=2\tR(x)$.

\medskip\noindent\textbf{(ii).}
If $x\in\cE\setminus\cA_\delta$, $x$ is strongly aligned, hence path (i) has resistance equal to $2\tR(x) \ge 2R_0 \ge 2$ according to \cref{lemma:ritel:noise_stability,cond:R_0}. Moreover, path (ii) has resistance equal to $2$, whereas path (iii) has resistance greater than $2$. Therefore, path (ii) implies an easy edge and $r^\star(x) = 2$.

\medskip\noindent\textbf{(iii).}
If $x\in\cE_\delta\setminus\cA_\delta$, path (ii) has resistance equal to $2$, hence $r^\star(x) \le 2$.

\medskip\noindent\textbf{(iv).}
If $x\in\cC\setminus\cE$, $x$ is strongly aligned, hence path (i) has resistance equal to $2\tR(x) \ge 2$. Path (ii) has resistance $2$ and path (iii) has resistance greater than $2-\tS_+(x)$, which itself is lower than $2$. Therefore, $r^\star(x) \ge 2-\tS_+(x)$.

\medskip\noindent\textbf{(v).}
If $x\in\cC_\delta\setminus\cE_\delta$, path (iii) has resistance lower than $2-\tS_-(x)$, hence $r^\star(x) \le 2-\tS_-(x)$.
\end{proof}

It remains to discuss the resistance of edges $D \to x\in\cC_\delta$.
\begin{lemma}
    \label{lemma:ritel:r_D}
    \hfill
    \begin{enumerate}[(i)]
        \item If $x\in\cC_\delta$, then $r(D\to x) \ge 1 - \tW(x)$.
        \item If $x\in\cC$, then $r(D\to x) = 1 - \tW(x)$.
    \end{enumerate}
\end{lemma}

\begin{proof}
Let $x=(\blm,\bla,\blv)\in\cC_\delta$.

\medskip\noindent\textbf{(i).}
In \ac{ritel} as in \ac{itel}, a content agent cannot decrease in benchmark utility without first becoming discontent.
As in \cref{lemma:itel:r_D}, it follows that for a discontent agent to accept the bin $\lv$, it is needed at some point that it accepts a bin $v\le\lv$ with resistance $F(v) \ge F(\lv)$. Applying the reasoning to all agents yields $r(D\to x) \ge \sum_i F(\lv_i) = 1-\tW(x)$.

\medskip\noindent\textbf{(ii).}
When $x$ is strongly aligned, the direct path $D\to x$ corresponding to all agents simultaneously playing according to $\bla$ and accepting $\blv$ has a resistance equal to $1-\tW(x)$, so that the lower bound becomes an equality in this case and $r(D\to x) = 1-\tW(x)$.
\end{proof}

Before moving on to computing potentials, notice that the previous lemmas allow us to conclude on the nature of the \ac{ritel} process.
First of all, states $x\in\cA_\delta$ are absorbing in the perturbed process if and only if $\tR(x)=\pinfty$ according to \cref{lemma:ritel:r_x}. Note that an outward resistance of $\pinfty$ does not guarantee in general that no path exists leaving the state, only that if such path exists the probability of it happening is negligible before $\epsilon^r$ for any $r\ge0$. In this context however, \cref{lemma:ritel:easy_edge} states that the path associated with making offset observations does not exist when $\tR(x)=\pinfty$, hence states in $\cA$ are indeed absorbing.
Moreover, \cref{lemma:ritel:r_x,lemma:ritel:r_D} show that there is a path in $P^\varepsilon$ with finite resistance from $D$ to any state $x\in\cC_\delta$, and back when $x$ is not absorbing. The key argument to derive the above is exactly the same as in \cref{sec:itel:proof_res}, that is, easy edges from states in $\cC_\delta$ cannot create cycles as they always increase benchmark welfare before eventually going to $D$.

Therefore, if $\cA\ne\emptyset$, there is a path from any state to a state in $\cA$ which is absorbing. Hence the perturbed process is not irreducible and its recurrence classes are exactly the singletons $\{x\}\subset\cA$.
On the other hand, when $\cA=\emptyset$, $D$ communicates with all states in $\cC_\delta$, so that all these states belong to the same recurrence class. Hence the perturbed process consists of a unique recurrence class, \ie, is irreducible\footnote{There may actually be states that are not be part of the recurrence class of $P^\varepsilon$. In this case we restrict ourselves to the study of the single recurrence class, which is a \ac{arpmp}. This is without loss of generality, as the algorithm ends up \ac{as} in this recurrence class regardless of the starting state.}. It is also aperiodic as the transition $D\to D$ has positive probability.
This concludes \cref{prop:ritel:recurrence_P_eps}.
When $\cA\ne\emptyset$, the process converges \ac{as} to one of its recurrence classes, which is one of the absorbing states $x\in\cA$.

The rest of this section is devoted to the case where $\cA=\emptyset$, where \cref{thm:sss_min_gamma_AR} can be applied.
We now study minimal rooted trees for each state, in order to compute their potential.
\begin{lemma}
    \label{lemma:ritel:potential}
    \hfill
    \begin{enumerate}[(i)]
        \item $\gamma(D) = \sum_{x\in\cC} r^\star(x)$.
        \item $\gamma(x) = \gamma(D) - r^\star(x) + r(D\to x)
        \begin{cases}
            \le \gamma(D) - 2\tR(x) + 1 - \tW(x) & \text{if $x\in\cA_\delta$,}\\
            = \gamma(D) - 1 -\tW(x) & \text{if $x\in\cE\setminus\cA_\delta$,}\\
            \ge \gamma(D) - 1 -\tW(x) & \text{if $x\in\cE_\delta\setminus\cA_\delta$,}\\
            \le \gamma(D) - 1 - \tW(x) + \tS_+(x) & \text{if $x\in\cC\setminus\cE$,}\\
            \ge \gamma(D) - 1 - \tW(x) + \tS_-(x) & \text{if $x\in\cC_\delta\setminus\cE_\delta$.}
        \end{cases}$
    \end{enumerate}
\end{lemma}

\begin{proof}
The exact same reasoning as in \cref{lemma:itel:potential} shows that a $D$-tree can be constructed using only easy edges, proving (i), and that $\gamma(x) = \gamma(D) - r^\star(x) + r(D\to x)$ for $x\in\cC_\delta$. The different cases are deduced by replacing $r^\star(x)$ and $r(D\to x)$ by the values given by \cref{lemma:ritel:r_x,lemma:ritel:r_D}, concluding (ii).
\end{proof}

We now identify states that may belong to $\cX^\star$ using the bounds on $F$ and $G$ from \cref{cond:itel} which are common to \ac{itel} and \ac{ritel} and imply that $1-F_0 \le \tW \le 1$, $1-G_0 < \tS_- \le 1$, and
\begin{align*}
    0 \le 1-\tW < \tS_- \le 1
\end{align*}
for any possible value of $\tS_-$ and $\tW$.
When $\cE\ne\emptyset$, the goal is to show that
\begin{align}
    \label{eq:proof:ritel:case1}
    \cX^\star \subset \cA_\delta \cup \La x\in\cE_\delta\setminus\cA_\delta ~:~ \tW(x) \ge \max_\cE \tW\Ra~.
\end{align}
Denote $\tW^\star = \max_{\cE}\tW$ and $x^\star\in\cE$ such that $\tW(x^\star) = \tW^\star$. \Cref{lemma:ritel:potential} states that
\begin{align*}
    \begin{cases}
        \gamma(x^\star) = \gamma(D) - 1-\tW^\star & \text{if $x^\star\notin\cA_\delta$,}\\
        \gamma(x^\star) \le \gamma(D) - 2\tR(x^\star) + 1-\tW^\star \le \gamma(D) - 1-\tW^\star & \text{if $x^\star\in\cA_\delta$,}
    \end{cases}
\end{align*}
as $x$ is strongly aligned hence $\tR(x) \ge R_0 \ge 1$ due to \cref{lemma:ritel:noise_stability,cond:R_0}.
It follows that states in $\cX^\star$ necessarily have a potential lower than $\gamma(D) - 1-\tW^\star$.
Let $x\in \cE_\delta\setminus\cA_\delta$ with $\tW(x) < \tW^\star$ and $y\in\cC_\delta\setminus\cE_\delta$. \Cref{lemma:ritel:potential} and the different bounds on $\tW$ and $\tS$ imply that
\begin{align*}
    \begin{cases}
        \gamma(x) \ge \gamma(D) - 1 - \tW(x)
        >  \gamma(D) - 1 - \tW^\star~,\\
        \gamma(y) \ge \gamma(D) - 1 - \tW(y) + \tS_-(y) \ge \gamma(D) - 2 + \tS_-(y) > \gamma(D) - 1 - \tW^\star~,\\
        \gamma(D) > \gamma(D) - 1-\tW^\star~.
    \end{cases}
\end{align*}
Then, a state in $\cX^\star$ cannot be in either of the above three cases as its potential would be higher than that of $x^\star$, hence \cref{eq:proof:ritel:case1} holds which concludes the first case of \cref{thm:ritel:X_star}.
When $\cE=\emptyset$, the goal is to show that
\begin{align}
    \label{eq:proof:ritel:case2}
    \cX^\star \subset \cE_\delta \cup \La x\in\cC_\delta\setminus\cE_\delta ~:~ \tW(x)-\tS_-(x) \ge \max_\cC (\tW-\tS_+)\Ra~.
\end{align}
Denote $\tW^\star-\tS_+^\star = \max_{x\in\cC} \tW(x)-\tS_+(x)$ and $x^\star\in\cC$ such that $\tW(x) = \tW^\star$ and $\tS_+(x) = \tS_+^\star$. \Cref{lemma:ritel:potential} states that
\begin{align*}
    \gamma(x^\star)
    \le \gamma(D) - 1 - \tW^\star + \tS_+^\star~.
\end{align*}
It follows that states in $\cX^\star$ necessarily have a potential lower than $\gamma(D) - 1 - \tW^\star + \tS_+^\star$. 
Let $y\in \cC_\delta\setminus\cE_\delta$ with $\tW(y)-\tS_-(y) < \tW^\star-\tS_+^\star$. \Cref{lemma:ritel:potential} and the different bounds on $\tW$ and $\tS$ imply that
\begin{align*}
    \begin{cases}
        \gamma(y) \ge \gamma(D) - 1 - \tW(y) + \tS_-(y) > \gamma(D) - 1 - \tW^\star + \tS_+^\star~,\\
        \gamma(D) > \gamma(D) - 1 - \tW^\star + \tS_+^\star~.
    \end{cases}
\end{align*}
Then, a state in $\cX^\star$ cannot be in either of the above two cases as its potential would be higher than that of $x^\star$, hence \cref{eq:proof:ritel:case2} holds which concludes the second case of \cref{thm:ritel:X_star}.

\subsection{Convergence in Terms of Action Profiles}
\label{sec:ritel:convergence_in_action_profiles}

It remains to translate \cref{thm:ritel:X_star} to a convergence result regarding action profiles as stated by \cref{thm:ritel:convergence}. We do so using \cref{lemma:ritel:correspondence_equilibrium,lemma:ritel:correspondence_tW&tS}.
First, let us discuss the properties of action profiles of states that can be part of $\cX^\star$ according to \cref{thm:ritel:X_star}.
If $\cA\ne\emptyset$, then the process converges \ac{as} to $\cA$. Since $\cA\subset\cA_\delta$, the benchmark action profiles of such states are $\delta$-optimal according to \cref{lemma:ritel:correspondence_equilibrium} (ii).

If $\cA=\emptyset$ and $\cE\ne\emptyset$ then
\[\cX^\star \subset \cA_\delta \cup \La x\in\cE_\delta\setminus\cA_\delta:\tW(x) \ge \tW(x^\star)\Ra\]
where $x^\star\in\cE$ is such that $\tW(x^\star) = \max_{x\in\cE}\tW(x)$.
Denote $\ba^\star$ an action profile maximizing $\tW$ among equilibrium action profiles and $y\in\cC$ the strongly aligned state associated with $\ba^\star$, which belongs to $\cE$ according to \cref{lemma:ritel:correspondence_equilibrium} (i). Then $\tW(x^\star) \ge \tW(y) \ge \tW_-(\ba^\star)$ according to \cref{lemma:ritel:correspondence_tW&tS} (ii).
Let $x\in\cX^\star$ and denote $\bla$ its action profile. If $x\in\cA_\delta$ then $\bla$ is $\delta$-optimal according to \cref{lemma:ritel:correspondence_equilibrium} (ii).
If $x\in\cE_\delta\setminus\cA_\delta$ and $\tW(x) \ge \tW(x^\star)$ then $\bla$ is a $3\delta$-equilibrium according to \cref{lemma:ritel:correspondence_equilibrium} (iii). Moreover, \cref{lemma:ritel:correspondence_tW&tS} (i) implies that $\tW_+(\bla) \ge \tW(x) \ge \tW(x^\star) \ge \tW_-(\ba^\star)$.
We conclude that when $\cE\ne\emptyset$, the action profiles $\bla$ of \acp{sss} are either $\delta$-optimal or $3\delta$-equilibria with $\tW_+(\bla) \ge \tW_-(\ba^\star)$.

If $\cE=\emptyset$ then
\[\cX^\star \subset \cE_\delta \cup \La x\in\cC_\delta\setminus\cE_\delta ~:~ \tW(x)-\tS_-(x) \ge \tW(x^\star)-\tS_+(x^\star)\Ra\]
where $x^\star\in\cC$ is such that $\tW(x^\star)-\tS_+(x^\star) = \max_{x\in\cC} \tW(x)-\tS_+(x)$.
Denote $\ba^\star$ an action profile maximizing $\tW-\tS$ among equilibria and $y\in\cC$ the strongly aligned state associated with $\ba^\star$. Then $\tW(x^\star)-\tS_+(x^\star) \ge \tW(y)-\tS_+(y) \ge \tW_-(\ba^\star)-\tS_+(\ba^\star)$ according to \cref{lemma:ritel:correspondence_tW&tS} (ii) and (iv).
Let $x\in\cX^\star$ and denote $\bla$ its action profile.
If $x\in\cE_\delta$ then $\bla$ is a $3\delta$-equilibrium according to \cref{lemma:ritel:correspondence_equilibrium} (iii).
if $x\in\cC_\delta\setminus\cE_\delta$ with $\tW(x)-\tS_-(x) \ge \tW(x^\star)-\tS_+(x^\star)$ then
\[\tW_+(\bla)-\tS_-(\bla)
\ge \tW(x)-\tS_-(x)
\ge \tW(x^\star)-\tS_+(x^\star)
\ge \tW_-(\ba^\star)-\tS_+(\ba^\star)\]
according to \cref{lemma:ritel:correspondence_tW&tS} (i) and (iii).
We conclude that when $\cE=\emptyset$, the action profiles $\bla$ of \acp{sss} either are $3\delta$-equilibria or satisfy $\tW_+(\bla)-\tS_-(\bla) \ge \tW_-(\ba^\star)-\tS_+(\ba^\star)$.

Now, depending on the existence of equilibrium action profiles, the following holds.
\begin{itemize}
    \item If there exist equilibrium action profiles, then $\cE\ne\emptyset$ and the \ac{sss} of the \ac{ritel} process have action profiles $\ba$ that are either $\delta$-optimal or $3\delta$-equilibria with $\tW_+(\ba) \ge \tW_-(\ba^\star)$ where $\ba^\star$ maximizes $\tW$ over equilibrium action profiles.
    \item Else, the \ac{sss} of the \ac{ritel} process have action profiles $\ba$ that either are $3\delta$-equilibria or satisfy $\tW_+(\ba)-\tS_-(\ba) \ge \tW_-(\ba^\star)-\tS_+(\ba^\star)$ where $\ba^\star$ maximizes $\tW-\tS$ over all action profiles.
\end{itemize}
Notice that the non-existence of equilibrium action profiles does not guarantee that $\cE=\emptyset$. The above statement remains true, as the action profiles described in the case of $\cE\ne\emptyset$ are included in the ones described in the case $\cE=\emptyset$.
Likewise, it is not possible to know whether $\cA=\emptyset$ or not solely from the mean utilities of action profiles, as this information requires the knowledge of the support of the utility distributions. However, the action profiles described in the case $\cA\ne\emptyset$ are included in the ones described in the other cases.
This concludes \cref{thm:ritel:convergence}.

\end{document}